\documentclass{article}

\PassOptionsToPackage{numbers, compress, sort}{natbib}

\usepackage[preprint]{neurips_2024}

\usepackage[utf8]{inputenc} 
\usepackage[T1]{fontenc}    
\usepackage{url}            
\usepackage{amssymb,amsfonts,amsmath,amsthm} 
\usepackage{enumerate,enumitem,graphicx,mathrsfs,eucal,verbatim, bbm, derivative}
\usepackage{tikz, tikz-cd,graphicx}
\usepackage{svg}
\usepackage{wasysym}
\usepackage{rotating}
\usepackage{hyperref}    
\usepackage{xcolor}
\hypersetup{
     colorlinks = true,
     linkcolor = blue,
     anchorcolor = blue,
     citecolor = red,
     filecolor = blue,
     urlcolor = blue
     }
\usepackage{graphicx}   
\usepackage{mathrsfs}
\usetikzlibrary{arrows}
\usepackage[small]{caption}
\usepackage{subcaption} 
\usepackage{wrapfig}
\usepackage{pgfplots}
\pgfplotsset{compat=1.15}
\usepackage{epigraph}
\usepackage{footnotehyper}
\makesavenoteenv{table}
\makesavenoteenv{tabular}

\theoremstyle{plain}
\newtheorem{thm}{Theorem}[section]

\newtheorem{fact}[thm]{Fact}

\newtheorem{prop}[thm]{Proposition}
\newtheorem{lem}[thm]{Lemma}
\newtheorem{conj}[thm]{Conjecture}
\newtheorem{quest}[thm]{Question}

\theoremstyle{definition}
\newtheorem{defn}{Definition}
\newtheorem{exmp}[thm]{Example}
\theoremstyle{remark}

\newcommand{\reg}[1]{#1_{\mathrm{reg}}}
\newcommand{\chow}[1]{\mathrm{CH}\left( #1\right)}
\newcommand{\tang}[2]{\mathrm{T}_{#1}{#2}}
\newcommand{\ed}[1]{\mathrm{ED}(#1)}
\newcommand{\numm}[1]{n_\mathrm{#1}}
\newcommand{\crit}{\mathrm{Crit}}
\newcommand{\grass}{\mathrm{Gr}}

\DeclareMathOperator{\degg}{\mathrm{deg}}

\DeclareMathOperator{\cspan}{\mathrm{colsp}}
\DeclareMathOperator{\rspan}{\mathrm{rowsp}}
\DeclareMathOperator{\spann}{\mathrm{span}}
\DeclareMathOperator{\rank}{\mathrm{rank}}

\newcommand{\RR}{\mathbb R}
\newcommand{\PP}{\mathbb P}
\newcommand{\CC}{\mathbb C}

\let\svthefootnote\thefootnote
\newcommand\freefootnote[1]{%
  \let\thefootnote\relax%
  \footnotetext{\hspace{-.5em}#1}%
  \let\thefootnote\svthefootnote%
}

\title{Critical Points of Degenerate Metrics \\on Algebraic Varieties: \\
A Tale of Overparametrization}

\author{%
Giovanni Luca Marchetti *\\
KTH Royal Institute of Technology
 \And
Erin Connelly *\\
University of Osnabr\"uck
\And 
Paul Breiding *\\
University of Osnabr\"uck
\And
Kathl\'en Kohn *\\
KTH Royal Institute of Technology \\ \& Digital Futures
}

\begin{document}

\maketitle

\begin{abstract}
    We study the critical points over an algebraic variety of an optimization problem defined by a quadratic objective that is degenerate. This scenario arises in machine learning when the dataset size is small with respect to the model, and is typically referred to as overparametrization. Our main result relates the degenerate optimization problem to a nondegenerate one via a projection. In the highly-degenerate regime, we find that a central role is played by the ramification locus of the projection. Additionally, we provide tools for counting the number of critical points over projective varieties, and discuss specific cases arising from deep learning. Our work bridges tools from algebraic geometry with ideas from machine learning, and it extends the line of literature around the Euclidean distance degree to the degenerate setting. 
\end{abstract}


\section{Introduction}\label{sec:intro}
\freefootnote{*Equal contribution.}Optimization of a quadratic objective over a constrained space is central in machine learning, and many other branches of the sciences and engineering.
Often, the constraints are polynomial equalities or inequalities, making the underlying space a (semi-)algebraic variety.
This has motivated the development of tools from algebraic geometry for analyzing quadratic optimization problems over algebraic varieties.
The behavior of generic\footnote{We will use the term \emph{generic} when referring to quantities lying outside of some proper algebraic subset.} points on algebraic varieties allows us to define global invariants related to quadratic optimization. An important example is the \emph{Euclidean distance degree} (EDD), introduced in the seminal article~\cite{draisma2016euclidean}. This invariant quantifies the number of critical points of a quadratic form on a variety.

A typical regression objective of this form in machine learning and statistical inference is the mean-squared error, consisting in minimizing the squared distance between the model and the data. This problem can be rephrased as the optimization of a data-dependent quadratic form over the space of functions or distributions parametrized by the model. This space is often referred to as \emph{neuromanifold} in the context of machine learning \cite{amari2001geometrical,marchetti2025invitationneuroalgebraicgeometry} or, more generally, as hypothesis space.
Neuromanifolds of polynomial models -- e.g., multilayer perceptrons with polynomial activation or un-normalized self-attention mechanisms -- are (semi-)algebraic varieties.
As arbitrary neuromanifolds can be approximated by algebraic ones, their geometry and optimization properties can be understood by taking the limit of the algebraic setting
\cite{marchetti2025invitationneuroalgebraicgeometry}. Based on this, several works have studied quadratic optimization over algebraic neuromanifolds, especially their EDD \cite{trager2019pure, kubjas2024geometry, shahverdi2024geometryoptimizationpolynomialconvolutional, shahverdi2025algebraic, shahverdi2025learning, henry2024geometry,arjevani2025geometry,gafvert2020computational}.

A fundamental assumption underlying the above-mentioned literature is that the quadratic form is nondegenerate, i.e., it coincides with a squared distance from a target point in the ambient space.
In deep learning, this holds when the dataset size is large enough, but fails in scenarios of data scarcity.
In the latter case, the optimization problem is underspecified: the number of available data points is insufficient to determine the target unambiguously, and the model optimizes an objective with, potentially, a continuum of global minima. From an algebraic perspective, the quadratic form defined by the objective is degenerate, with rank proportional to the dataset size. The behavior of this optimization problem differs drastically from the nondegenerate case when the dataset size is significantly lower than the dimension of the neuromanifold or, more roughly, than the number of parameters, which is typical in contemporary large-scale models. This scenario is often referred to as  \emph{overparametrization} \cite{belkin2021fit}. 

Motivated by the above, in this work, we consider the problem of optimizing a \emph{degenerate} quadratic form over an algebraic variety. We analyze the critical points of the quadratic form over the smooth locus of the variety, generalizing the EDD to the degenerate scenario. 
The following example serves as an illustration of our setting and our results.

\begin{exmp} \label{ex:attention}
    The \emph{self-attention mechanism} is the key ingredient of the popular Transformer architecture \cite{vaswani2017attention}.  
    Its un-normalized variant -- often referred to as `linear' \cite{kleiman1974transversality} or `lightning' \cite{henry2024geometry} self-attention --  parametrizes cubic functions of the form 
\begin{align} \label{eq:attention}
        \mathbb{R}^{e \times t} &\to \mathbb{R}^{e'  \times t},\quad
    M \mapsto V MM^\top A M.
\end{align}
Here, the learnable weights are the entries of the \emph{value matrix} $V \in \mathbb{R}^{e' \times e}$ and the \emph{attention matrix} $A \in \mathbb{R}^{e \times e}$, where $A$ is of rank at most $a \leq e$ as it arises as the product of the so-called \emph{key} and \emph{query matrices}.
For fixed architecture parameters $e,e',a,t$, the neuromanifold is the set of all cubic functions \eqref{eq:attention} for varying $V$ and $A$.
Since \eqref{eq:attention} depends linearly on the entries of $V$ and $A$, the neuromanifold is a semi-algebraic set (see Section \ref{sect:algebraic_sets}), meaning that it can be described by polynomial equalities and inequalities inside the ambient space of all cubic functions in $M$.

For instance, for $e'=1$ and $e = t = a = 2$, it was shown in \cite{henry2024geometry} that the neuromanifold can be seen as a hypersurface in $\mathbb{R}^6$ with coordinates $(c_1,c_2,\ldots,c_6)$, where it satisfies the equation
\begin{equation}\label{attention_poly}
    c_1^2c_6^2 + c_4^2c_3^2 + c_1c_3c_5^2 + c_2^2c_4c_6 - 2c_1c_4c_3c_6 - c_2c_1c_6c_5 - c_2c_4c_3c_5 =0.
\end{equation}
A slice of that hypersurface with a generic 3-dimensional affine subspace is shown in Figure~\ref{fig:attention}.

We are interested in finding a point $c$ on the hypersurface that is closest to some point $u \in \mathbb{R}^6$, where closeness is measured by a positive semi-definite quadratic form $Q$.
For almost all choices of $u$ and positive definite $Q$, this problem has 14 complex critical points.
This number is the EDD of the hypersurface (more precisely, the \emph{generic} EDD; we explain the difference in Section~\ref{sec:quadratic_optimization}).
When the quadratic form $Q$ is degenerate, meaning that its associated bilinear form has a kernel $K$, the behavior can be different. 
For instance, for almost all choices of $u$ and positive semi-definite $Q$ with a 2-dimensional kernel, there are 4 isolated complex critical points plus a continuous curve of complex critical points.
For other fixed kernel dimensions $k := \dim  K$, the complex critical points are generically as follows: 
\begin{center}
\begin{tabular}{c|c}
$k = \dim  K$ & complex critical point set \\ \hline
$0$ & $14$ points\\
$1$ & $14$ points\\
$2$ & $4$ points $+$ a curve \\
$3$ & a surface \\
$4$ &a $3$-dimensional subvariety\\
$5$ & a $4$-dimensional subvariety\\
\end{tabular}
\end{center}
The continuous sets of critical points for $k \geq 2$ all arise from intersecting the hypersurface in \eqref{attention_poly} with the affine subspace $K+u$.
We discuss this example in detail in Section~\ref{sec:attention}.
\hfill $\diamondsuit$
\end{exmp}

\begin{figure}[!h]
    \centering
    \includegraphics[width=0.275\linewidth]{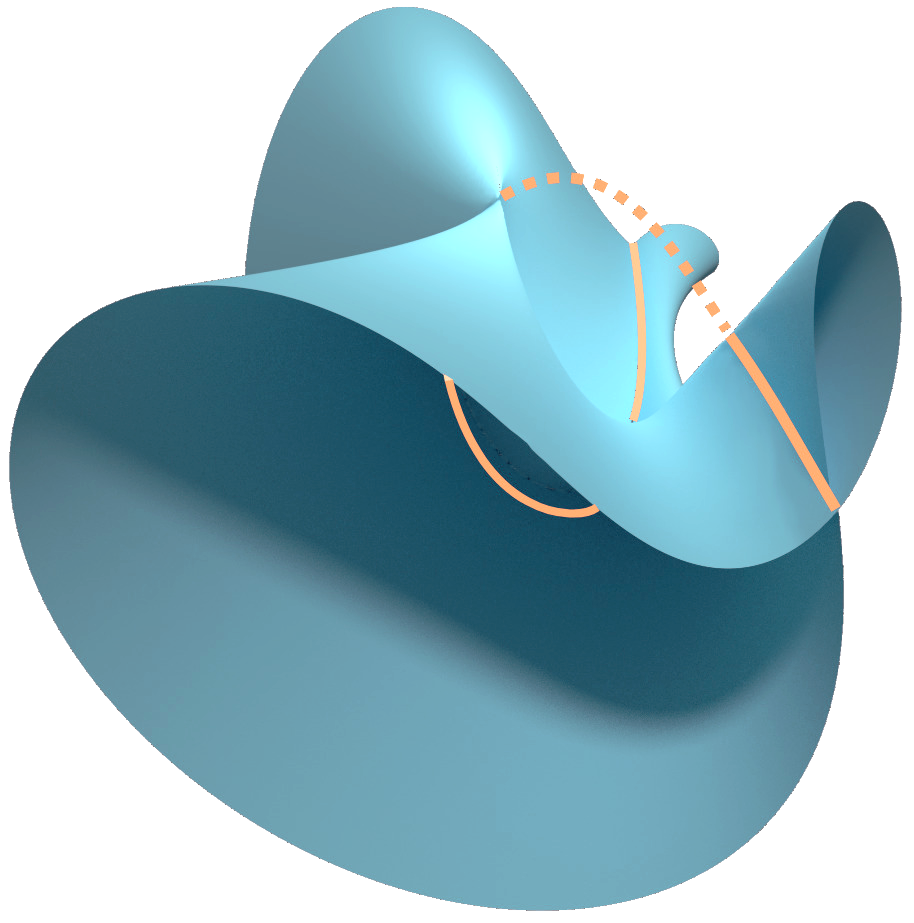}
    \caption{A slice of the neuromanifold of lightning self-attention mechanisms (adapted from \cite{henry2024geometry}). 
    }
    \label{fig:attention}
\end{figure}

\subsection{Overview of Main Results} \label{sec:overviewResults}
Here, we explain our results in a friendly way. Technical definitions are provided in the next section. 

We consider a (possibly nonlinear) algebraic variety $X \subset \RR^n$. Its dimension is denoted 
\begin{equation}\label{def:d}d := \dim  X.\end{equation}
While the technical definition of dimension is provided in the next section, intuitively one can think of $X$ being ``almost'' a manifold, meaning that locally around almost all points of $X$ it has manifold structure. These points are called smooth points; they form the smooth locus of $X$, and  the dimension of each of those manifolds is $d$. 

The main object of study in this paper are the (complex) critical points of a quadratic form $Q(x-u)$, where $x$ ranges over the smooth locus of $X$ and $u \in \RR^n$ is a given point in the ambient space. 
When the form is nondegenerate, the number of these complex critical points is the same for almost all $u$; this number is the EDD of the variety $X$ and the quadric $Q$ \cite{draisma2016euclidean}. 
We instead focus on the degenerate case, meaning the bilinear form associated with $Q$ has a kernel~$K$, whose dimension quantifies the degeneracy of $Q$. We denote that dimension
\begin{equation}\label{def:k}k := \dim  K.\end{equation}

We show that the case of degenerate quadrics $Q$ can be reduced, in a subtle way, to the nondegenerate setting via the orthogonal projection $\pi \colon \RR^n \to K^\perp$ onto the orthogonal complement of~$K$. The subtlety is caused by a paradigm shift that depends on the degeneracy $k$ of $Q$ compared to the codimension $n-d$ of~$X$. Note that $\pi$ turns $Q$ into a nondegenerate quadratic form  on $K^\perp$. The set of critical points of $\pi$ over $X$ is called the \emph{ramification locus}. 
Its image under that projection (i.e., the set of critical values of projecting $X$) is known as the \emph{branch locus}. Both are shaded in darker blue in Figures~\ref{fig:ellPlane}--\ref{fig:ellLineOutside}.  

\begin{thm}[Main Result, informal version]\label{main_informal}
For almost all\footnote{More precisely, \emph{almost all} refers to generic objects.} quadratic forms~$Q$ with a $k$-dimensional kernel and almost all $u \in \RR^n$, 
the critical points $x$ of $Q(x-u)$ satisfy the following:

\textbf{Case 1:} $k < n -d$\\(i.e., the quadric is mildly degenerate, its kernel has dimension less than the codimension of $X$).

In this case, the critical points  are in bijection with the ones  of the projected nondegenerate quadric over the projection of $X$; see Figure~\ref{fig:curve}.  In particular,  they are of finite cardinality, which is the EDD of the projection of $X$ (Theorem~\ref{prop:generalXsmallK}). 

\textbf{Case 2:} $k \geq n-d$\\ (i.e., the quadric is significantly degenerate, its kernel has dimension at least the codimension of $X$).

Here, there are two classes of critical points:
\begin{enumerate}
    \item All points in the intersection $(K + u) \cap X$ (red in Figures~\ref{fig:ellPlane} and~\ref{fig:ellLineInside}). These are the points $x$ on $X$ with zero loss, i.e.,  $Q(x-u)=0$.
    \item Finitely many points lying on the ramification locus (yellow in Figures~\ref{fig:ellPlane}--\ref{fig:ellLineOutside}). 
    
    If $X$ is sufficiently general\footnote{The precise assumptions are stated in Section \ref{sec:veryDegenerate}, right before Lemma~\ref{lem:projBirational}. }, these points are in bijection with the critical points of the projected nondegenerate quadric over the branch locus (Theorem \ref{prop:generalXlargeK}).
    As such, they are counted by the EDD of the branch locus.
\end{enumerate}
\end{thm}

\begin{figure}
    \centering
     \begin{subfigure}[b]{0.48\textwidth}
         \centering
        \tikzset{every picture/.style={line width=0.75pt}} 

\begin{tikzpicture}[x=0.75pt,y=0.75pt,yscale=-.7,xscale=.7]

\draw [color={rgb, 255:red, 154; green, 202; blue, 214 }  ,draw opacity=1 ][line width=1.5]    (160.13,170.63) .. controls (246.55,162.54) and (274.01,65.68) .. (230.38,89.32) ;
\draw [color={rgb, 255:red, 154; green, 202; blue, 214 }  ,draw opacity=1 ][line width=1.5]    (225.04,91.91) .. controls (200.98,109.21) and (199.23,127.11) .. (209.43,134.07) .. controls (213.71,136.99) and (220.1,137.99) .. (227.83,136.2) ;
\draw  [draw opacity=0][fill={rgb, 255:red, 238; green, 238; blue, 238 }  ,fill opacity=1 ][dash pattern={on 4.5pt off 4.5pt}] (162.92,223.3) -- (377.79,223.3) -- (317.49,280.24) -- (102.62,280.24) -- cycle ;
\draw [color={rgb, 255:red, 154; green, 202; blue, 214 }  ,draw opacity=1 ][line width=1.5]    (235.87,134.17) .. controls (251.92,131.43) and (280.11,123.18) .. (293.99,117.08) ;
\draw [color={rgb, 255:red, 66; green, 66; blue, 66 }  ,draw opacity=1 ]   (227.21,143.79) -- (227.21,191.36) ;
\draw [color={rgb, 255:red, 154; green, 202; blue, 214 }  ,draw opacity=1 ][line width=1.5]    (156.7,274.41) .. controls (175.57,273.44) and (235.51,265.39) .. (245.58,243.43) .. controls (255.65,221.46) and (196.7,222.19) .. (204.53,244.5) .. controls (212.36,266.81) and (271.91,273.4) .. (292.2,272.53) ;
\draw  [draw opacity=0][fill={rgb, 255:red, 66; green, 66; blue, 66 }  ,fill opacity=1 ] (222.85,240.85) .. controls (222.85,239.42) and (224.28,238.25) .. (226.04,238.25) .. controls (227.8,238.25) and (229.23,239.42) .. (229.23,240.85) .. controls (229.23,242.29) and (227.8,243.45) .. (226.04,243.45) .. controls (224.28,243.45) and (222.85,242.29) .. (222.85,240.85) -- cycle ;
\draw  [draw opacity=0][fill={rgb, 255:red, 255; green, 189; blue, 96 }  ,fill opacity=1 ] (244.14,110.37) .. controls (244.14,108.61) and (245.56,107.19) .. (247.33,107.19) .. controls (249.09,107.19) and (250.51,108.61) .. (250.51,110.37) .. controls (250.51,112.14) and (249.09,113.56) .. (247.33,113.56) .. controls (245.56,113.56) and (244.14,112.14) .. (244.14,110.37) -- cycle ;
\draw  [draw opacity=0][fill={rgb, 255:red, 255; green, 189; blue, 96 }  ,fill opacity=1 ] (200.72,240.14) .. controls (200.72,238.71) and (202.15,237.55) .. (203.91,237.55) .. controls (205.67,237.55) and (207.1,238.71) .. (207.1,240.14) .. controls (207.1,241.58) and (205.67,242.74) .. (203.91,242.74) .. controls (202.15,242.74) and (200.72,241.58) .. (200.72,240.14) -- cycle ;
\draw  [draw opacity=0][fill={rgb, 255:red, 255; green, 189; blue, 96 }  ,fill opacity=1 ] (217.34,95.03) .. controls (217.34,93.27) and (218.77,91.84) .. (220.53,91.84) .. controls (222.29,91.84) and (223.72,93.27) .. (223.72,95.03) .. controls (223.72,96.79) and (222.29,98.22) .. (220.53,98.22) .. controls (218.77,98.22) and (217.34,96.79) .. (217.34,95.03) -- cycle ;
\draw  [draw opacity=0][fill={rgb, 255:red, 255; green, 189; blue, 96 }  ,fill opacity=1 ] (200.57,120.31) .. controls (200.57,118.55) and (202,117.13) .. (203.76,117.13) .. controls (205.52,117.13) and (206.95,118.55) .. (206.95,120.31) .. controls (206.95,122.08) and (205.52,123.5) .. (203.76,123.5) .. controls (202,123.5) and (200.57,122.08) .. (200.57,120.31) -- cycle ;
\draw  [draw opacity=0][fill={rgb, 255:red, 255; green, 189; blue, 96 }  ,fill opacity=1 ] (218.22,227.85) .. controls (218.22,226.41) and (219.65,225.25) .. (221.41,225.25) .. controls (223.17,225.25) and (224.6,226.41) .. (224.6,227.85) .. controls (224.6,229.28) and (223.17,230.44) .. (221.41,230.44) .. controls (219.65,230.44) and (218.22,229.28) .. (218.22,227.85) -- cycle ;
\draw  [draw opacity=0][fill={rgb, 255:red, 255; green, 189; blue, 96 }  ,fill opacity=1 ] (243.49,238.2) .. controls (243.49,236.76) and (244.92,235.6) .. (246.68,235.6) .. controls (248.44,235.6) and (249.87,236.76) .. (249.87,238.2) .. controls (249.87,239.63) and (248.44,240.8) .. (246.68,240.8) .. controls (244.92,240.8) and (243.49,239.63) .. (243.49,238.2) -- cycle ;
\draw [color={rgb, 255:red, 66; green, 66; blue, 66 }  ,draw opacity=1 ]   (227.21,40.44) -- (227.21,132.89) ;
\draw  [draw opacity=0][fill={rgb, 255:red, 66; green, 66; blue, 66 }  ,fill opacity=1 ] (223.96,55.71) .. controls (223.96,53.95) and (225.38,52.52) .. (227.14,52.52) .. controls (228.91,52.52) and (230.33,53.95) .. (230.33,55.71) .. controls (230.33,57.47) and (228.91,58.9) .. (227.14,58.9) .. controls (225.38,58.9) and (223.96,57.47) .. (223.96,55.71) -- cycle ;
\draw  [color={rgb, 255:red, 196; green, 196; blue, 196 }  ,draw opacity=1 ] (100.12,10.1) -- (380.02,10.1) -- (380.02,290) -- (100.12,290) -- cycle ;


\end{tikzpicture}
         \caption{1-dimensional kernel $K$ (black). The critical points on the space curve (yellow) are in bijection with the critical points (yellow) after projecting onto the plane $K^\perp$ (grey).}
         \label{fig:curve}
        \end{subfigure}
        \hfill 
        \begin{subfigure}[b]{0.48\textwidth}
         \centering
          
\tikzset {_t7jb8j9v5/.code = {\pgfsetadditionalshadetransform{ \pgftransformshift{\pgfpoint{89.1 bp } { -108.9 bp }  }  \pgftransformscale{1.32 }  }}}
\pgfdeclareradialshading{_bn4kvcu91}{\pgfpoint{-72bp}{88bp}}{rgb(0bp)=(1,1,1);
rgb(0bp)=(1,1,1);
rgb(25bp)=(0.31,0.66,0.75);
rgb(400bp)=(0.31,0.66,0.75)}
\tikzset{every picture/.style={line width=0.75pt}} 

\begin{tikzpicture}[x=0.75pt,y=0.75pt,yscale=-.7,xscale=.7]

\draw [color={rgb, 255:red, 238; green, 238; blue, 238 }  ,draw opacity=1 ][line width=1.5]    (115.75,250.02) -- (147.81,250.02) ;
\draw [color={rgb, 255:red, 238; green, 238; blue, 238 }  ,draw opacity=1 ][line width=1.5]    (333.48,250.48) -- (365.54,250.48) ;
\draw [color={rgb, 255:red, 154; green, 202; blue, 214 }  ,draw opacity=1 ][line width=2.25]    (149.37,250.48) -- (333.48,250.48) ;
\draw  [draw opacity=0][shading=_bn4kvcu91,_t7jb8j9v5] (149.51,115.54) .. controls (149.61,90.97) and (190.9,71.22) .. (241.74,71.42) .. controls (292.59,71.62) and (333.72,91.71) .. (333.62,116.28) .. controls (333.52,140.85) and (292.23,160.6) .. (241.39,160.4) .. controls (190.54,160.2) and (149.41,140.11) .. (149.51,115.54) -- cycle ;
\draw  [color={rgb, 255:red, 2; green, 136; blue, 165 }  ,draw opacity=1 ][fill={rgb, 255:red, 255; green, 189; blue, 96 }  ,fill opacity=1 ] (145.33,115.54) .. controls (145.33,113.92) and (146.64,112.61) .. (148.26,112.61) .. controls (149.88,112.61) and (151.19,113.92) .. (151.19,115.54) .. controls (151.19,117.16) and (149.88,118.47) .. (148.26,118.47) .. controls (146.64,118.47) and (145.33,117.16) .. (145.33,115.54) -- cycle ;
\draw  [color={rgb, 255:red, 2; green, 136; blue, 165 }  ,draw opacity=1 ][fill={rgb, 255:red, 255; green, 189; blue, 96 }  ,fill opacity=1 ] (330.55,115.55) .. controls (330.55,113.94) and (331.86,112.62) .. (333.48,112.62) .. controls (335.1,112.62) and (336.41,113.94) .. (336.41,115.55) .. controls (336.41,117.17) and (335.1,118.49) .. (333.48,118.49) .. controls (331.86,118.49) and (330.55,117.17) .. (330.55,115.55) -- cycle ;
\draw  [draw opacity=0][fill={rgb, 255:red, 255; green, 136; blue, 136 }  ,fill opacity=1 ] (238.73,250.43) .. controls (238.73,248.81) and (240.04,247.5) .. (241.66,247.5) .. controls (243.28,247.5) and (244.59,248.81) .. (244.59,250.43) .. controls (244.59,252.05) and (243.28,253.37) .. (241.66,253.37) .. controls (240.04,253.37) and (238.73,252.05) .. (238.73,250.43) -- cycle ;
\draw  [color={rgb, 255:red, 2; green, 136; blue, 165 }  ,draw opacity=1 ][fill={rgb, 255:red, 255; green, 189; blue, 96 }  ,fill opacity=1 ] (144.88,250.28) .. controls (144.88,248.66) and (146.19,247.34) .. (147.81,247.34) .. controls (149.43,247.34) and (150.74,248.66) .. (150.74,250.28) .. controls (150.74,251.9) and (149.43,253.21) .. (147.81,253.21) .. controls (146.19,253.21) and (144.88,251.9) .. (144.88,250.28) -- cycle ;
\draw  [color={rgb, 255:red, 2; green, 136; blue, 165 }  ,draw opacity=1 ][fill={rgb, 255:red, 255; green, 189; blue, 96 }  ,fill opacity=1 ] (330.55,250.74) .. controls (330.55,249.12) and (331.86,247.81) .. (333.48,247.81) .. controls (335.1,247.81) and (336.41,249.12) .. (336.41,250.74) .. controls (336.41,252.36) and (335.1,253.67) .. (333.48,253.67) .. controls (331.86,253.67) and (330.55,252.36) .. (330.55,250.74) -- cycle ;
\draw  [draw opacity=0][fill={rgb, 255:red, 238; green, 238; blue, 238 }  ,fill opacity=1 ] (212.64,67.4) -- (264.82,18.26) -- (264.82,72.94) .. controls (257.98,72.06) and (250.84,71.55) .. (243.48,71.46) .. controls (242.79,71.2) and (242.1,71.06) .. (241.39,71.06) .. controls (232.86,71.06) and (225.95,91.06) .. (225.95,115.73) .. controls (225.95,140.4) and (232.86,160.4) .. (241.39,160.4) .. controls (241.39,160.4) and (241.4,160.4) .. (241.4,160.4) .. controls (245.03,160.4) and (248.6,160.29) .. (252.12,160.09) -- (212.64,197.27) -- (212.64,67.4) -- cycle ;
\draw  [draw opacity=0][fill={rgb, 255:red, 66; green, 66; blue, 66 }  ,fill opacity=1 ] (238.73,55.33) .. controls (238.73,53.71) and (240.04,52.4) .. (241.66,52.4) .. controls (243.28,52.4) and (244.59,53.71) .. (244.59,55.33) .. controls (244.59,56.95) and (243.28,58.26) .. (241.66,58.26) .. controls (240.04,58.26) and (238.73,56.95) .. (238.73,55.33) -- cycle ;
\draw [color={rgb, 255:red, 255; green, 136; blue, 136 }  ,draw opacity=1 ][line width=1.5]    (241.74,71.42) .. controls (219.42,72.14) and (221.07,158.66) .. (241.27,160.48) ;
\draw [color={rgb, 255:red, 255; green, 136; blue, 136 }  ,draw opacity=1 ][line width=1.5]  [dash pattern={on 1.69pt off 2.76pt}]  (242.71,71.39) .. controls (262.81,73.89) and (260.7,159.19) .. (242.23,160.45) ;
\draw  [color={rgb, 255:red, 196; green, 196; blue, 196 }  ,draw opacity=1 ] (102.1,8.1) -- (382,8.1) -- (382,288) -- (102.1,288) -- cycle ;


\end{tikzpicture}
         \caption{2-dimensional kernel $K$ (grey) intersects the surface in 1-dimensional set of zero-loss solutions (red). Ramification/branch locus consists of 2 points, both being critical points (yellow).}
         \label{fig:ellPlane}
     \end{subfigure} 

     \vspace{1em}

     \begin{subfigure}[t]{0.48\textwidth}
         \centering
        \tikzset {_ywwjew6e3/.code = {\pgfsetadditionalshadetransform{ \pgftransformshift{\pgfpoint{89.1 bp } { -108.9 bp }  }  \pgftransformscale{1.32 }  }}}
\pgfdeclareradialshading{_st2na5o78}{\pgfpoint{-72bp}{88bp}}{rgb(0bp)=(1,1,1);
rgb(0bp)=(1,1,1);
rgb(25bp)=(0.31,0.66,0.75);
rgb(400bp)=(0.31,0.66,0.75)}
\tikzset{every picture/.style={line width=0.75pt}} 

\begin{tikzpicture}[x=0.75pt,y=0.75pt,yscale=-.7,xscale=.7]

\draw  [draw opacity=0][shading=_st2na5o78,_ywwjew6e3] (143.87,110.55) .. controls (143.96,85.98) and (185.26,66.22) .. (236.1,66.43) .. controls (286.94,66.63) and (328.08,86.72) .. (327.98,111.29) .. controls (327.88,135.86) and (286.59,155.61) .. (235.74,155.41) .. controls (184.9,155.2) and (143.77,135.12) .. (143.87,110.55) -- cycle ;
\draw  [draw opacity=0][fill={rgb, 255:red, 238; green, 238; blue, 238 }  ,fill opacity=1 ][dash pattern={on 4.5pt off 4.5pt}] (162.63,214.75) -- (380.59,214.75) -- (318.45,273.14) -- (100.48,273.14) -- cycle ;
\draw [color={rgb, 255:red, 2; green, 136; blue, 165 }  ,draw opacity=1 ][line width=1.5]    (143.87,110.55) .. controls (145.35,142.74) and (323.98,141) .. (327.84,112.77) ;
\draw [color={rgb, 255:red, 2; green, 136; blue, 165 }  ,draw opacity=1 ][line width=1.5]  [dash pattern={on 1.69pt off 2.76pt}]  (143.87,110.55) .. controls (152.11,78.89) and (323.21,85.06) .. (327.84,112.77) ;
\draw [color={rgb, 255:red, 66; green, 66; blue, 66 }  ,draw opacity=1 ] [dash pattern={on 0.84pt off 2.51pt}]  (236.1,66.43) -- (235.74,155.41) ;
\draw [color={rgb, 255:red, 66; green, 66; blue, 66 }  ,draw opacity=1 ][line width=0.75]    (236.1,34.02) -- (236.1,66.43) ;
\draw [color={rgb, 255:red, 66; green, 66; blue, 66 }  ,draw opacity=1 ]   (235.74,155.41) -- (235.74,184.99) ;
\draw  [color={rgb, 255:red, 2; green, 136; blue, 165 }  ,draw opacity=1 ][fill={rgb, 255:red, 177; green, 224; blue, 236 }  ,fill opacity=1 ][line width=1.5]  (143.72,243.17) .. controls (143.72,230.65) and (184.94,220.49) .. (235.78,220.49) .. controls (286.62,220.49) and (327.84,230.65) .. (327.84,243.17) .. controls (327.84,255.7) and (286.62,265.86) .. (235.78,265.86) .. controls (184.94,265.86) and (143.72,255.7) .. (143.72,243.17) -- cycle ;
\draw  [draw opacity=0][fill={rgb, 255:red, 255; green, 136; blue, 136 }  ,fill opacity=1 ] (233.17,66.43) .. controls (233.17,64.81) and (234.48,63.5) .. (236.1,63.5) .. controls (237.72,63.5) and (239.03,64.81) .. (239.03,66.43) .. controls (239.03,68.05) and (237.72,69.36) .. (236.1,69.36) .. controls (234.48,69.36) and (233.17,68.05) .. (233.17,66.43) -- cycle ;
\draw  [draw opacity=0][fill={rgb, 255:red, 255; green, 136; blue, 136 }  ,fill opacity=1 ] (232.81,155.41) .. controls (232.81,153.79) and (234.13,152.48) .. (235.74,152.48) .. controls (237.36,152.48) and (238.68,153.79) .. (238.68,155.41) .. controls (238.68,157.03) and (237.36,158.34) .. (235.74,158.34) .. controls (234.13,158.34) and (232.81,157.03) .. (232.81,155.41) -- cycle ;
\draw  [draw opacity=0][fill={rgb, 255:red, 255; green, 189; blue, 96 }  ,fill opacity=1 ] (140.93,109.3) .. controls (140.93,107.68) and (142.25,106.37) .. (143.87,106.37) .. controls (145.48,106.37) and (146.8,107.68) .. (146.8,109.3) .. controls (146.8,110.92) and (145.48,112.23) .. (143.87,112.23) .. controls (142.25,112.23) and (140.93,110.92) .. (140.93,109.3) -- cycle ;
\draw  [draw opacity=0][fill={rgb, 255:red, 255; green, 189; blue, 96 }  ,fill opacity=1 ] (324.91,109.31) .. controls (324.91,107.69) and (326.22,106.38) .. (327.84,106.38) .. controls (329.46,106.38) and (330.77,107.69) .. (330.77,109.31) .. controls (330.77,110.93) and (329.46,112.25) .. (327.84,112.25) .. controls (326.22,112.25) and (324.91,110.93) .. (324.91,109.31) -- cycle ;
\draw  [draw opacity=0][fill={rgb, 255:red, 255; green, 136; blue, 136 }  ,fill opacity=1 ] (233.17,242.66) .. controls (233.17,241.04) and (234.48,239.73) .. (236.1,239.73) .. controls (237.72,239.73) and (239.03,241.04) .. (239.03,242.66) .. controls (239.03,244.28) and (237.72,245.59) .. (236.1,245.59) .. controls (234.48,245.59) and (233.17,244.28) .. (233.17,242.66) -- cycle ;
\draw  [draw opacity=0][fill={rgb, 255:red, 255; green, 189; blue, 96 }  ,fill opacity=1 ] (140.79,241.93) .. controls (140.79,240.31) and (142.1,238.99) .. (143.72,238.99) .. controls (145.34,238.99) and (146.66,240.31) .. (146.66,241.93) .. controls (146.66,243.55) and (145.34,244.86) .. (143.72,244.86) .. controls (142.1,244.86) and (140.79,243.55) .. (140.79,241.93) -- cycle ;
\draw  [draw opacity=0][fill={rgb, 255:red, 255; green, 189; blue, 96 }  ,fill opacity=1 ] (324.91,241.93) .. controls (324.91,240.31) and (326.22,238.99) .. (327.84,238.99) .. controls (329.46,238.99) and (330.77,240.31) .. (330.77,241.93) .. controls (330.77,243.55) and (329.46,244.86) .. (327.84,244.86) .. controls (326.22,244.86) and (324.91,243.55) .. (324.91,241.93) -- cycle ;
\draw  [draw opacity=0][fill={rgb, 255:red, 66; green, 66; blue, 66 }  ,fill opacity=1 ] (233.24,47.46) .. controls (233.24,45.84) and (234.56,44.53) .. (236.17,44.53) .. controls (237.79,44.53) and (239.11,45.84) .. (239.11,47.46) .. controls (239.11,49.08) and (237.79,50.39) .. (236.17,50.39) .. controls (234.56,50.39) and (233.24,49.08) .. (233.24,47.46) -- cycle ;
\draw  [draw opacity=0][fill={rgb, 255:red, 255; green, 189; blue, 96 }  ,fill opacity=1 ] (233.17,133.56) .. controls (233.17,131.94) and (234.48,130.62) .. (236.1,130.62) .. controls (237.72,130.62) and (239.03,131.94) .. (239.03,133.56) .. controls (239.03,135.18) and (237.72,136.49) .. (236.1,136.49) .. controls (234.48,136.49) and (233.17,135.18) .. (233.17,133.56) -- cycle ;
\draw  [draw opacity=0][fill={rgb, 255:red, 255; green, 189; blue, 96 }  ,fill opacity=1 ] (233.17,88.8) .. controls (233.17,87.18) and (234.48,85.87) .. (236.1,85.87) .. controls (237.72,85.87) and (239.03,87.18) .. (239.03,88.8) .. controls (239.03,90.42) and (237.72,91.74) .. (236.1,91.74) .. controls (234.48,91.74) and (233.17,90.42) .. (233.17,88.8) -- cycle ;
\draw  [draw opacity=0][fill={rgb, 255:red, 255; green, 189; blue, 96 }  ,fill opacity=1 ] (233.17,265.81) .. controls (233.17,264.19) and (234.48,262.88) .. (236.1,262.88) .. controls (237.72,262.88) and (239.03,264.19) .. (239.03,265.81) .. controls (239.03,267.43) and (237.72,268.74) .. (236.1,268.74) .. controls (234.48,268.74) and (233.17,267.43) .. (233.17,265.81) -- cycle ;
\draw  [draw opacity=0][fill={rgb, 255:red, 255; green, 189; blue, 96 }  ,fill opacity=1 ] (233.17,220.29) .. controls (233.17,218.67) and (234.48,217.35) .. (236.1,217.35) .. controls (237.72,217.35) and (239.03,218.67) .. (239.03,220.29) .. controls (239.03,221.91) and (237.72,223.22) .. (236.1,223.22) .. controls (234.48,223.22) and (233.17,221.91) .. (233.17,220.29) -- cycle ;
\draw  [color={rgb, 255:red, 196; green, 196; blue, 196 }  ,draw opacity=1 ] (100.06,3.1) -- (379.96,3.1) -- (379.96,283) -- (100.06,283) -- cycle ;


\end{tikzpicture}
         \caption{1-dimensional kernel $K$ (black) intersects the surface at 2 zero-loss solutions (red). Ramification/branch locus is a curve (darker blue), containing 4 critical points (yellow).}
         \label{fig:ellLineInside}
     \end{subfigure}
     \hfill 
     \begin{subfigure}[t]{0.48\textwidth}
         \centering
        \tikzset {_3yh4m3lve/.code = {\pgfsetadditionalshadetransform{ \pgftransformshift{\pgfpoint{89.1 bp } { -108.9 bp }  }  \pgftransformscale{1.32 }  }}}
\pgfdeclareradialshading{_ysapazymc}{\pgfpoint{-72bp}{88bp}}{rgb(0bp)=(1,1,1);
rgb(0bp)=(1,1,1);
rgb(25bp)=(0.31,0.66,0.75);
rgb(400bp)=(0.31,0.66,0.75)}
\tikzset{every picture/.style={line width=0.75pt}} 

\begin{tikzpicture}[x=0.75pt,y=0.75pt,yscale=-.7,xscale=.7]

\draw [color={rgb, 255:red, 66; green, 66; blue, 66 }  ,draw opacity=1 ]   (343.73,41.21) -- (343.38,192.18) ;
\draw  [draw opacity=0][shading=_ysapazymc,_3yh4m3lve] (143.3,117.55) .. controls (143.4,92.98) and (184.7,73.22) .. (235.54,73.43) .. controls (286.38,73.63) and (327.52,93.72) .. (327.42,118.29) .. controls (327.32,142.86) and (286.02,162.61) .. (235.18,162.41) .. controls (184.34,162.2) and (143.2,142.12) .. (143.3,117.55) -- cycle ;
\draw  [draw opacity=0][fill={rgb, 255:red, 238; green, 238; blue, 238 }  ,fill opacity=1 ][dash pattern={on 4.5pt off 4.5pt}] (162.07,221.75) -- (380.03,221.75) -- (317.88,280.14) -- (99.92,280.14) -- cycle ;
\draw [color={rgb, 255:red, 2; green, 136; blue, 165 }  ,draw opacity=1 ][line width=1.5]    (143.3,117.55) .. controls (144.79,149.74) and (323.42,148) .. (327.28,119.77) ;
\draw [color={rgb, 255:red, 2; green, 136; blue, 165 }  ,draw opacity=1 ][line width=1.5]  [dash pattern={on 1.69pt off 2.76pt}]  (143.3,117.55) .. controls (151.54,85.89) and (322.65,92.06) .. (327.28,119.77) ;
\draw  [color={rgb, 255:red, 2; green, 136; blue, 165 }  ,draw opacity=1 ][fill={rgb, 255:red, 177; green, 224; blue, 236 }  ,fill opacity=1 ][line width=1.5]  (143.16,250.17) .. controls (143.16,237.65) and (184.38,227.49) .. (235.22,227.49) .. controls (286.06,227.49) and (327.28,237.65) .. (327.28,250.17) .. controls (327.28,262.7) and (286.06,272.86) .. (235.22,272.86) .. controls (184.38,272.86) and (143.16,262.7) .. (143.16,250.17) -- cycle ;
\draw  [draw opacity=0][fill={rgb, 255:red, 66; green, 66; blue, 66 }  ,fill opacity=1 ] (340.82,54.46) .. controls (340.82,52.84) and (342.13,51.53) .. (343.75,51.53) .. controls (345.37,51.53) and (346.69,52.84) .. (346.69,54.46) .. controls (346.69,56.08) and (345.37,57.39) .. (343.75,57.39) .. controls (342.13,57.39) and (340.82,56.08) .. (340.82,54.46) -- cycle ;
\draw  [color={rgb, 255:red, 196; green, 196; blue, 196 }  ,draw opacity=1 ] (99.5,10.1) -- (379.4,10.1) -- (379.4,290) -- (99.5,290) -- cycle ;
\draw  [draw opacity=0][fill={rgb, 255:red, 66; green, 66; blue, 66 }  ,fill opacity=1 ] (338.34,250.17) .. controls (338.34,248.56) and (339.66,247.24) .. (341.28,247.24) .. controls (342.9,247.24) and (344.21,248.56) .. (344.21,250.17) .. controls (344.21,251.79) and (342.9,253.11) .. (341.28,253.11) .. controls (339.66,253.11) and (338.34,251.79) .. (338.34,250.17) -- cycle ;
\draw  [draw opacity=0][fill={rgb, 255:red, 255; green, 189; blue, 96 }  ,fill opacity=1 ] (140.23,250.17) .. controls (140.23,248.56) and (141.54,247.24) .. (143.16,247.24) .. controls (144.78,247.24) and (146.09,248.56) .. (146.09,250.17) .. controls (146.09,251.79) and (144.78,253.11) .. (143.16,253.11) .. controls (141.54,253.11) and (140.23,251.79) .. (140.23,250.17) -- cycle ;
\draw  [draw opacity=0][fill={rgb, 255:red, 255; green, 189; blue, 96 }  ,fill opacity=1 ] (324.34,250.17) .. controls (324.34,248.56) and (325.66,247.24) .. (327.28,247.24) .. controls (328.9,247.24) and (330.21,248.56) .. (330.21,250.17) .. controls (330.21,251.79) and (328.9,253.11) .. (327.28,253.11) .. controls (325.66,253.11) and (324.34,251.79) .. (324.34,250.17) -- cycle ;
\draw  [draw opacity=0][fill={rgb, 255:red, 255; green, 189; blue, 96 }  ,fill opacity=1 ] (140.37,117.55) .. controls (140.37,115.93) and (141.68,114.62) .. (143.3,114.62) .. controls (144.92,114.62) and (146.23,115.93) .. (146.23,117.55) .. controls (146.23,119.17) and (144.92,120.48) .. (143.3,120.48) .. controls (141.68,120.48) and (140.37,119.17) .. (140.37,117.55) -- cycle ;
\draw  [draw opacity=0][fill={rgb, 255:red, 255; green, 189; blue, 96 }  ,fill opacity=1 ] (324.34,116.84) .. controls (324.34,115.22) and (325.66,113.91) .. (327.28,113.91) .. controls (328.9,113.91) and (330.21,115.22) .. (330.21,116.84) .. controls (330.21,118.46) and (328.9,119.77) .. (327.28,119.77) .. controls (325.66,119.77) and (324.34,118.46) .. (324.34,116.84) -- cycle ;


\end{tikzpicture}
         \caption{1-dimensional kernel $K$ (black) does not intersect surface over $\RR$, so the global minimizer and maximizer (yellow) lie on the ramification locus.
     }
         \label{fig:ellLineOutside}
     \end{subfigure}
    \caption{Illustration of different scenarios, where $X$ is either a curve (a) or a surface (b), (c) or (d) in $\RR^3$. } 
    \label{fig:curveEllipsoid} 
\end{figure}
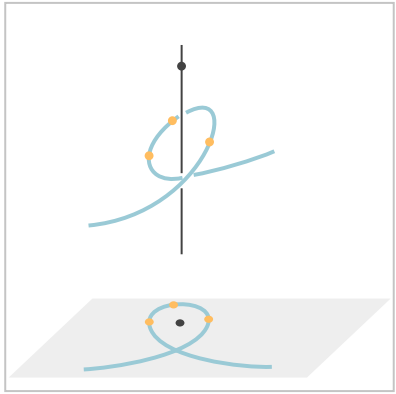
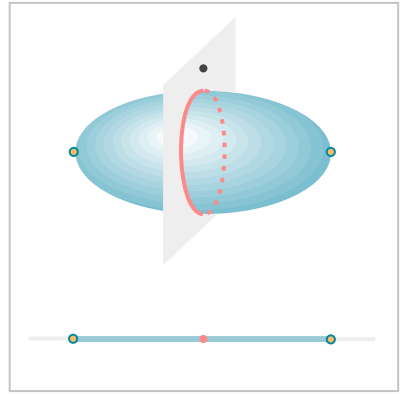
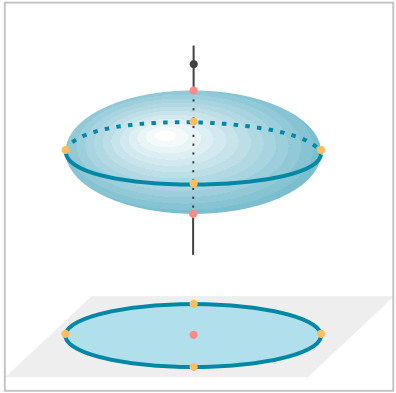
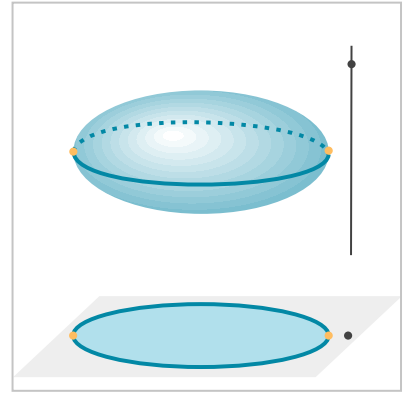

In the case $k > n-d$, the set of zero-loss solutions, by a dimension count, is expected to be positive-dimensional and continuous, as in Figure~\ref{fig:ellPlane} and Example~\ref{ex:attention}.\enlargethispage{\baselineskip}
If that set contains real solutions, it comprises the global minima. 
However, depending on the geometry of $X$, it can be empty over $\RR$ (even with positive probability over $u$), as in Figure~\ref{fig:ellLineOutside}.
In that case, the global minima of $Q(x-u)$ must be located on the ramification locus of projecting $X$.
This yields an interesting phenomenon: when the quadric is significantly degenerate, the optimization problem exhibits a bias towards the ramification locus.
 Note that this locus only depends on the kernel $K$ (not on $Q$ itself) and in particular not on $u$. However, when $u$ varies, then  the finitely many critical points on the ramification locus vary within that locus, cf.\ Figure~\ref{fig:ellLineOutside}.

Theorem \ref{main_informal} does not hold for every quadratic form, but only for sufficiently generic ones.
For a particular choice of quadric, the critical point sets may look different.
We illustrate this in
Section~\ref{sec:attention}, where we revisit the self-attention Example~\ref{ex:attention}, and in 
Section~\ref{sec:determinant}, where we investigate the \emph{determinantal variety} $X$ consisting of matrices with bounded rank together with a standard degenerate quadric $Q$.
The latter scenario is relevant in machine learning, since it corresponds to minimizing the mean-squared error over a multilayer perceptron with no activation function.

Finally, we are interested in determining the cardinalities of the finite sets of critical points described in Theorem~\ref{main_informal}, i.e., the EDD of the projection of $X$ (for $k < n-d)$ and the EDD of the branch locus of projecting $X$ (for $k \geq n-d$).
This is simpler if $X$ is defined by  homogeneous equations (as in Example~\ref{ex:attention}) or, equivalently, it is a \emph{projective variety}.     
Then, the EDD of the projection of $X$, for $k < n-d$, coincides with the EDD of $X$ itself \cite[Corollary~6.1]{draisma2016euclidean},
which in turn equals the sum of the \emph{polar degrees} of $X$
\cite[Theorem~5.4]{draisma2016euclidean}.
The latter are classical invariants of projective varieties related to Chern classes \cite[Chapter~4]{mag}. We generalize this result, by showing that the EDD of the branch locus (for $k \geq n-d$) is a \emph{partial} sum of the polar degrees (Theorem \ref{thm:polardegs}).

\bigskip
\section{Background}\label{sec:background}
The purpose of this section is to provide the background needed to motivate and prove Theorem \ref{main_informal}. For this reason, we give short introductions into algebraic sets (Section \ref{sect:algebraic_sets}), the algebraic structure of quadratic optimization on such sets (Section \ref{sec:quadratic_optimization}), and finally into why optimization in machine learning often means optimization with degenerate quadrics (Section \ref{sec:machinepersp}). 

\subsection{Algebraic sets}\label{sect:algebraic_sets}

A real or complex \emph{algebraic variety} $X$ is a subset of $\mathbb{K}^n$ that is the solution set of finitely many polynomials in $\mathbb{K}[x_1,\ldots,x_n]$, where $\mathbb{K}$ is either $\RR$ or $\CC$.
In this article, we assume that $X$ is \emph{irreducible}, meaning that it cannot be written as the union of two proper non-empty subvarieties. If a variety is not irreducible, it is a finite union of irreducible varieties, so we can always consider these irreducible components individually. 

To define the singular locus of $X$, we consider its \emph{prime ideal} $I_X$ that consists of all polynomials in $\mathbb{K}[x_1,\ldots,x_n]$ that vanish along $X$ and fix a generating set $\langle f_1, \ldots, f_s \rangle = I_X$.
We then compute the $s \times n$ Jacobian matrix $J$ whose $(i,j)$-th entry is $\frac{\partial f_i}{\partial x_j}$.
For almost every point $p$ on $X$ (i.e., except for $p$ on some proper subvariety) the rank of $J(p)$ is the same. 
That rank is the \emph{codimension} $c$ of~$X$. The dimension is then 
\begin{equation}d  = \dim X = n-c.\end{equation}
The points $p$ on $X$ where the Jacobian $J(p)$ attains that rank are called \emph{smooth} or \emph{regular}. The remaining points on $X$ (where the rank of the Jacobian is less than the codimension of $X$) are the \emph{singular} ones.
They form a subvariety of $X$, defined by the prime ideal $I_X$ and the $c \times c$ minors of the Jacobian. We denote by $\reg{X}$ the set of smooth points of $X$.  

\begin{exmp}\label{ex:attentionSing}
    The attention hypersurface from Example~\ref{ex:attention} defined by the polynomial \eqref{attention_poly} is singular precisely at the points $c \in \RR^6$ where the matrix $\left[ \begin{smallmatrix}
        c_1 & c_2 & c_3 \\ c_4 & c_5 & c_6
    \end{smallmatrix}\right]$
    has rank at most one. This four-dimensional subvariety is depicted as the orange curve in the slice in Figure~\ref{fig:attention}. \hfill $\diamondsuit$
\end{exmp}

A \emph{semialgebraic set} in $\RR^n$ is a finite union of subsets of $\RR^n$, each defined by polynomial equalities and inequalities.
The Tarski--Seidenberg theorem \cite[Theorem~4.17]{michalek2021invitation} implies that the image of every polynomial map between real algebraic varieties is a semialgebraic set. We refer to \cite[Section 4]{BREIDING2023374}
for a detailed discussion on the structure of images of polynomial maps.
In particular, when a machine learning model parametrizes functions that are polynomial both in their input and in the parameters, then the set of parametrized functions (i.e., the neuromanifold) is a semialgebraic subset of a finite-dimensional vector space.

\begin{exmp}\label{exmp:determinantal}
    A multilayer perceptron parametrizes functions of the form 
    \begin{align} 
    \label{eq:MLP}
        \alpha_L \circ \sigma \circ \alpha_{L-1} \circ \sigma \circ \cdots \circ \sigma \circ \alpha_1,
    \end{align}
    where the $\alpha_i \colon \RR^{n_{i-1}} \to \RR^{n_i}$ are learnable affine linear functions and $\sigma$ is a nonlinear activation function that gets applied elementwise.
    If $\sigma$ is a polynomial of degree $\delta$, then the end-to-end function $\RR^{n_0} \to \RR^{n_L}$ in \eqref{eq:MLP} consists of $n_L$ polynomials of degree  at most $D:=\delta^{L-1}$.
    Thus, the neuromanifold is a semialgebraic set
    \begin{equation}X\subset (\RR[x_1, \ldots, x_{n_0}]_{\leq D})^{n_L}.\end{equation}
    Here, the subscript ``$\leq D$'' means multivariate polynomials of degree at most $D$. 

    In the case where all $\alpha_i$ are linear and $\sigma$ is the identity, the neuromanifold is actually a variety (i.e., polynomial inequalities are not needed when describing it). It is the \emph{determinantal variety} consisting of all linear maps $\RR^{n_0} \to \RR^{n_L}$ of rank at most $r := \min \{ n_0,n_1,\ldots,n_L \}$. In terms of matrices, this corresponds to: 
    \begin{align*}
        X := \{ W \in \RR^{n_L \times n_0} \mid \textnormal{all $(r+1)\times (r+1)$-minors of $W$ vanish} \},
    \end{align*}
    which we discuss in Section~\ref{sec:determinant}.
    \hfill $\diamondsuit$
\end{exmp}

The \emph{Zariski closure} of a subset $S$ in $\mathbb{K}^n$ is the smallest algebraic variety containing $S$.
For a semialgebraic set $S$, taking the Zariski closure essentially means to forget its defining inequalities and only keep the equalities.

\begin{exmp} \label{ex:attentionInequ}
    We have already seen in Example~\ref{ex:attention} that lightning self-attention mechanisms form a semialgebraic neuromanifold. In fact, for the concrete example $e'=1$ and $e = t = a = 2$, the neuromanifold is a 5-dimensional semialgebraic subset of $\RR^6$ that is defined by the polynomial equation \eqref{attention_poly} and the inequalities
    \begin{align}\label{eq:attentionInequalities}
        c_2^2 - 4c_1c_3 \geq 0 \quad \textnormal{ and } \quad c_5^2 - 4c_4c_6 \geq 0.
    \end{align}
    In Figure~\ref{fig:attention}, these inequalities cut off the dashed part of the orange curve. 
    The Zariski closure of the neuromanifold is defined by the single equation \eqref{attention_poly}.
    This algebraic variety does contain the dashed part of the orange curve in Figure~\ref{fig:attention}.
    Moreover, this means that the neuromanifold has boundary points in the Euclidean topology that its Zariski closure inherits from the ambient $\RR^6$. The boundary is described by the equation \eqref{attention_poly} of the Zariski closure together with the vanishing of the polynomials in \eqref{eq:attentionInequalities}. It is a 3-dimensional subset of the singular locus described in Example~\ref{ex:attentionSing}. We explain how to derive all (in)equalities in Section~\ref{sec:attention}.
    \hfill$\diamondsuit$
\end{exmp}

\subsection{Quadratic optimization}\label{sec:quadratic_optimization}

In this article, we study the optimization of a quadratic form over a real algebraic variety $X \subset \RR^n$, focusing on the critical points. Since optimization is often performed via gradient descent, the critical points are the most interesting ones, corresponding to the equilibria of the gradient flow. As customary in the EDD literature, we consider only the smooth locus $\reg{X}$. While singular and (Zariski) boundary points play an important role in some machine learning scenarios \cite{shahverdi2025learning} (cf.\ Example \ref{ex:attentionSing}), we leave the consideration of singularities and semialgebraic sets for future work. Thus, we wish to analyze the critical points of the map 
\begin{align}
\label{eq:quadricDistance}
    \reg{X} \to \RR, \quad x \mapsto Q(x-u),
\end{align}
where $u \in \RR^n$ and $Q$ is a quadratic form on $\RR^n$.\footnote{By abuse of notation, we will denote by $Q$ both the quadratic form and its associated bilinear form.}
In other words, we are interested in the smooth points $x \in \reg{X}$ such that $Q(v, x- u) = 0$ for all $v$ in the tangent space $\tang{x}{X}$. We point out that $\reg{X}$ is an open set in both the Euclidean or the Zariski topology on $X$. 

If one wants to find all real critical points, it is in general hard to work only over the real numbers. This is because the number of critical points changes with varying $u$ or $Q$, and so one typically does not have a certificate for when all critical points have been found.
However, over the complex numbers, we have a globally defined count of critical points, which enables us to find all real critical points by first finding all complex ones and then disregarding the non-real solutions.

To explain this formally, we consider the Zariski closure $X(\CC) \subseteq \CC^n$ of the real variety $X$ when viewed as a subset of $\CC^n$ (i.e., $X(\CC)$ is the set of all complex points satisfying the same polynomial equations as $X$). Similarly, we extend $Q$ to a quadratic form on $\CC^n$ (without complex conjugation).
Our main object of interest is therefore the set of complex critical points, denoted by
\begin{equation} \label{eq:crit}
    \crit_{X,Q}(u) = \{ x \in \reg{X}(\CC) \mid\forall v\in  \tang{x}{X} : Q(v, x- u) = 0 \}. 
\end{equation}
Geometrically, $x \in \crit_{X,Q}(u)$ if, and only if, $x - u$ is orthogonal w.r.t.\ $Q$ to the tangent space $\tang{x}{X}$
(e.g., see the normal lines in Figure~\ref{ex:GeneralCircle}). We  denote this relation 
\begin{equation}\label{eq:crit2}
(x - u) \perp_Q \tang{x}{X} \quad :\Longleftrightarrow\quad\forall v\in  \tang{x}{X} : Q(v, x- u) = 0.\end{equation}  

If the quadric $Q$ is nondegenerate, then, for generic $u \in \RR^n$, the sets $\crit_{X,Q}(u)$ have the same finite (and positive) cardinality \cite{draisma2016euclidean}. This invariant is known as the \emph{Euclidean distance degree} (EDD) of~$X$ w.r.t.\ $Q$, denoted by $\textnormal{EDD}_Q(X)$. 
For all $u \in \CC^n$, we either have that $|\crit_{X,Q}(u)|\leq \textnormal{EDD}_Q(X)$ (with equality when counted with multiplicity) or $|\crit_{X,Q}(u)| = \infty$.
Hence, the EDD provides, for generic $u \in \CC^n$, a certificate for having found all complex critical points.

Moreover, for generic quadrics $Q$, the degrees $\textnormal{EDD}_Q(X)$ are the same.
This invariant is called the \emph{generic Euclidean distance degree} of $X$, which we denote by $\textnormal{gEDD}(X)$. 
All nondegenerate quadrics $Q$ satisfy 
\begin{equation}\textnormal{EDD}_Q(X) \leq \textnormal{gEDD}(X)\end{equation}
due to semicontinuity \cite[Theorem 3.18]{mag}. The following standard example shows that this relation is in general indeed an inequality.

\begin{exmp}
    For the circle $X=\{(x,y) \mid x^2+y^2=1\}$ and a generic point $u$, the critical points with respect to the standard quadric $Q(x,y)=x^2+y^2$ are exactly the two intersection points between $X$ and the line $l=\{tu  \mid  t\in\RR\}$. However, for a more generic quadric $Q$, such as~$Q(x,y)=4x^2+y^2$ used in Figure~\ref{ex:GeneralCircle}, we obtain four critical points.
    \hfill $\diamondsuit$

    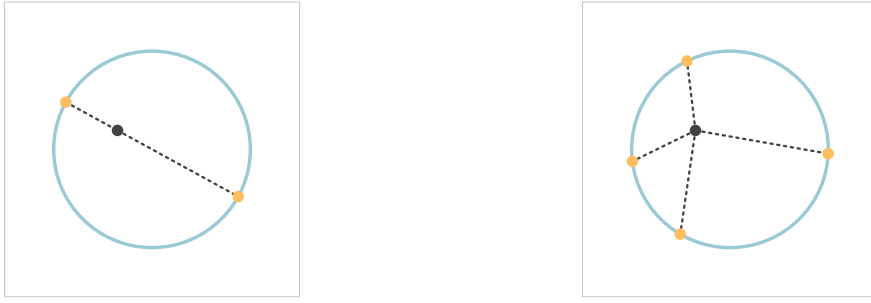
\begin{figure}
    \begin{subfigure}[b]{0.45\textwidth}
    \centering
    \definecolor{ffbdwq}{rgb}{1, 0.74117, 0.376470}
\definecolor{zacaff}{rgb}{0.60392,0.79215, 0.839215}
\definecolor{afard}{rgb}{0.25882, 0.25882, 0.25882}
\resizebox{.65\textwidth}{!}{
\begin{tikzpicture}[line cap=round,line join=round,>=triangle 45,x=1.5cm,y=1.5cm]
\draw [line width=1.5pt,color=zacaff] (0,0) circle (1.5cm);
\draw [line width=1pt,dash pattern=on 1pt off 2pt, color=afard] (-0.8768137851947905,0.48083010106729357)-- (-0.3509899183391754,0.192833619863706);
\draw [line width=1pt,dash pattern=on 1pt off 2pt, color=afard] (-0.3509899183391754,0.192833619863706)-- (0.8769780683804816,-0.48053040234686417);
\begin{scriptsize}
\draw [fill=afard, draw=none] (-0.3509899183391754,0.192833619863706) circle (2.5pt);
\draw [fill=ffbdwq, draw=none] (0.8769780683804816,-0.48053040234686417) circle (2.5pt);
\draw [fill=ffbdwq, draw=none] (-0.8768137851947905,0.48083010106729357) circle (2.5pt);
\end{scriptsize}
\draw  [color={rgb, 255:red, 196; green, 196; blue, 196 }  ,draw opacity=1 ] (-1.5,-1.5) -- (1.5,-1.5) -- (1.5,1.5) -- (-1.5,1.5) -- cycle ;
\end{tikzpicture}
}
    \caption{
    $Q(x,y)=x^2+y^2$ yields
     $\textnormal{EDD}_Q(X)=2$.}
    \end{subfigure}
    \hfill
    \begin{subfigure}[b]{0.45\textwidth}
    \centering
        \definecolor{ffbdwq}{rgb}{1, 0.74117, 0.376470}
\definecolor{zacaff}{rgb}{0.60392,0.79215, 0.839215}
\definecolor{afard}{rgb}{0.25882, 0.25882, 0.25882}
\resizebox{.65\textwidth}{!}{
\begin{tikzpicture}[line cap=round,line join=round,>=triangle 45,x=1.5cm,y=1.5cm]
\draw [line width=1.5pt,color=zacaff] (0,0) circle (1.5cm);
\draw [line width=1pt,dash pattern=on 1pt off 2pt, color=afard] (-0.3509899183391754,0.192833619863706)-- (-0.43676886150645855,0.8995737666352616);
\draw [line width=1pt,dash pattern=on 1pt off 2pt, color=afard] (-0.3509899183391754,0.192833619863706)-- (0.9990415112288157,-0.043772809387154686);
\draw [line width=1pt,dash pattern=on 1pt off 2pt, color=afard] (-0.3509899183391754,0.192833619863706)-- (-0.9925778520430143,-0.12161088616433992);
\draw [line width=1pt,dash pattern=on 1pt off 2pt, color=afard] (-0.3509899183391754,0.192833619863706)-- (-0.5056572982440202,-0.862734430014566);
\begin{scriptsize}
\draw [fill=afard, draw=none] (-0.3509899183391754,0.192833619863706) circle (2.5pt);
\draw [fill=ffbdwq, draw=none] (-0.9925778520430143,-0.12161088616433992) circle (2.5pt);
\draw [fill=ffbdwq, draw=none] (-0.5056572982440202,-0.862734430014566) circle (2.5pt);
\draw [fill=ffbdwq, draw=none] (0.9990415112288157,-0.043772809387154686) circle (2.5pt);
\draw [fill=ffbdwq, draw=none] (-0.43676886150645855,0.8995737666352616) circle (2.5pt);
\end{scriptsize}
\draw  [color={rgb, 255:red, 196; green, 196; blue, 196 }  ,draw opacity=1 ] (-1.5,-1.5) -- (1.5,-1.5) -- (1.5,1.5) -- (-1.5,1.5) -- cycle ;
\end{tikzpicture}
}
        \caption{$Q(x,y)=4x^2+y^2$ yields
     $\textnormal{EDD}_Q(X)=4$.}\label{ex:GeneralCircle}
        \end{subfigure}
    \caption{Critical points (yellow) and normal lines (dashed) on the circle w.r.t.\ two different quadrics~$Q$.}
    \end{figure}
\end{exmp}

Our main result Theorem~\ref{main_informal} extends the generic Euclidean distance degree of $X$ from nondegenerate quadrics to degenerate ones. 
As mentioned in Section~\ref{main_informal}, 
it does not hold for the $\textnormal{EDD}_Q(X)$ for arbitrary quadrics $Q$. 

\subsection{Machine Learning Perspective}\label{sec:machinepersp}
In this section, we explain how quadratic optimization arises in machine learning. To this end, we focus on polynomial models, meaning that the neuromanifold $X$ consists of polynomial functions~$\RR^{\numm{in}} \to \RR^{\numm{out}}$ of degree at most $D$. In other words, $X$ is contained in the vector space of all multivariate polynomials of degree at most $D$: 
\begin{equation}X\subset (\RR[x_1,\ldots,x_{\numm{in}}]_{\leq D})^{\numm{out}}\end{equation}
(compare Examples \ref{ex:attention} and \ref{exmp:determinantal}).
Given a finite dataset $\mathcal{S} \subseteq \RR^{\numm{in}} \times \RR^{\numm{out}}$, a typical regression problem is to minimize, over $f \in X$, the mean-squared error loss
\begin{equation}\label{eq:losss}
    \mathcal{L}(f) = \sum_{(a,b)\in \mathcal{S}} \Vert f(a)-b \Vert^2. 
\end{equation}
The optimization problem in \eqref{eq:losss} is equivalent to minimizing a (potentially degenerate) distance from some given point in the ambient vector space to $X$.
To see this,
we use the standard trick of turning polynomial regression into a linear one by considering the Veronese embedding $\nu_D$ that sends $(x_1,\ldots,x_{\numm{in}})$ to a tuple of all monomials in $(x_1,\ldots,x_{\numm{in}})$ of degree at most $D$.
Then, we can identify $f$ with its coefficient matrix $W$ whose $i$-th row contains the coefficients of the $i$-th coordinate function such that $f(a) = W \nu_D(a)$. 
Let us denote by $A$ and $B$ the matrices whose columns are the $\nu_D(a)$ and $b$, respectively, where $(a,b)$ runs over the whole dataset $\mathcal{S}$. 
\begin{exmp}
    For a single input $\numm{in}=1$, the Veronese embedding is $\nu_D(x) = (1,x,x^2,\ldots,x^D)$. Therefore, $A$ is a Vandermonde matrix. \hfill $\diamondsuit$
\end{exmp}
With the notation above, the mean-squared error loss \eqref{eq:losss} can be written in the concise form 
\begin{equation}\label{eq:losss_alternative}
    \mathcal{L}(f) = \mathcal{L}(W) = \Vert WA - B\Vert_{\textnormal{Frob}}^2,
\end{equation}
where the norm here is the \emph{Frobenius norm}  $\Vert M\Vert_{\textnormal{Frob}}=\sqrt{\mathrm{tr}(MM^\top)}$.

\begin{prop} \label{prop:MSE} Let $X \subseteq (\RR[x_1,\ldots,x_{\numm{in}}]_{\leq D})^{\numm{out}}$. Then, 
    \begin{equation}\label{eq:lossopt}
        \arg\min_{W \in X} \mathcal{L}(W) = \arg\min_{W \in X} \Vert W - BA^\dagger \Vert^2_{AA^\top},
    \end{equation}
    where 
   $A^\dagger$ is the Moore--Penrose inverse of $A$ and \begin{equation}\Vert M \Vert_{AA^\top}  :=  \sqrt{\textnormal{tr}(M A A^\top  M^\top)} = \| MA \|_{\textnormal{Frob}}.\end{equation}
\end{prop}
\begin{proof} This follows from the property $A^\dagger AA^\top = A^\top$ of the Moore--Penrose inverse. Writing `const.' for terms that do not depend on $W$, we see that
    \begin{align*}
        \Vert W - BA^\dagger \Vert^2_{AA^\top}
        &= \textnormal{tr}(WAA^\top W^\top) - 2 \,\textnormal{tr}(B A^\dagger AA^\top W^\top) + \textnormal{const.} \\
        &= \textnormal{tr}(WAA^\top W^\top) - 2 \,\textnormal{tr}(B A^\top W^\top) + \textnormal{const.}
        = \Vert WA-B \Vert_{\textnormal{Frob}}^2+ \textnormal{const.}
    \end{align*}
    Now the claim follows using \eqref{eq:losss_alternative}.
\end{proof}

Proposition~\ref{prop:MSE} says that minimizing the mean-squared error loss over a polynomial model $X$ amounts to finding the closest point on $X$ to a given point $u: = BA^\dagger$, where closeness is measured by the Frobenius-type seminorm induced by $AA^\top$. The quadratic form over $(\RR[x_1,\ldots,x_{\numm{in}}]_{\leq D})^{\numm{out}}$ corresponding to this seminorm is the Kronecker product 
\begin{equation}\label{special_Q}Q:=I_{\numm{out}} \otimes AA^\top.\end{equation}
Both $AA^\top$ and $Q=I_{\numm{out}} \otimes AA^\top$ are rank-deficient whenever the dataset size $ |\mathcal{S}| $ is less than the dimension of $\RR[x_1,\ldots,x_{\numm{in}}]_{\leq D}$. As mentioned in Section \ref{sec:intro}, this degenerate scenario is common in modern large-scale machine learning models, and is referred to as overparametrization \cite{belkin2021fit}. This is the core motivation behind our work.

We note that, even in the setting of noisy and big data $|\mathcal{S}| \geq \dim \RR[x_1,\ldots,x_{\numm{in}}]_{\leq D}$, the quadric in \eqref{special_Q} is not generic, but of a special form: 1) it has tensor structure, and 2) the matrix $AA^\top$ is structured, as seen in Example~\ref{exampleAAt} below. This means that our main result (Theorem \ref{main_informal}), which assumes a generic quadric, does not immediately hold in this setting. 
\begin{exmp}\label{exampleAAt}
    Consider quadratic polynomial functions in two variables, i.e., $\numm{in}=2$ and $D=2$. 
    Choosing the monomial ordering 
    $\nu_2(x_1,x_2) = (1,x_1,x_2,x_1^2,x_1x_2,x_2^2)$, we obtain 
    \begin{align*}
        AA^\top = \sum_{((a_1,a_2),b) \in \mathcal{S}} \begin{pmatrix}
            1 & a_1 & a_2 & a_1^2 & a_1a_2 & a_2^2 \\
            a_1 & a_1^2 & a_1a_2 & a_1^3 & a_1^2a_2 & a_1a_2^2 \\
            a_2 & a_1a_2 & a_2^2 & a_1^2a_2 & a_1a_2^2 & a_2^3 \\
            a_1^2 & a_1^3 & a_1^2a_2 & a_1^4 & a_1^3a_2 & a_1^2a_2^2 \\
            a_1a_2 & a_1^2a_2 & a_1a_2^2 & a_1^3a_2 & a_1^2a_2^2 & a_1a_2^3 \\
            a_2^2 & a_1a_2^2 & a_2^3 & a_1^2a_2^2 & a_1a_2^3 & a_2^4 
        \end{pmatrix}.
    \end{align*}
    In particular, even when $|\mathcal{S}| \gg 6$, the matrix $AA^\top$ is \emph{not} an arbitrary symmetric positive definite matrix, as it has repeated entries. 
    \hfill $\diamondsuit$
\end{exmp}
\begin{quest}
    What is the semialgebraic description of the set of all matrices $AA^\top$ for varying data~$\mathcal{S}$ and fixed $\numm{in}, D, |\mathcal{S}|$?
\end{quest}
Despite the special structure of the quadric $Q$ in \eqref{special_Q}, we can still apply Theorem \ref{main_informal} using the following trick, provided the variety $X$ that we minimize the loss over is general. Consider a  general affine linear transformation $\varphi(x) = Mx+\beta$, and define the quadric $Q' := M^{-\top} Q M^{-1}$ and the variety $X' := \varphi(X)=\{\varphi(x)\mid x\in X\}$. Since~$\varphi$ is affine linear, $\tang{x'}{X'} = M\cdot \tang{x}{X},$ where $x'=\varphi(x)$. Therefore, the criticality condition  \eqref{eq:crit2} for $X$ and $X'$ are related by 
\begin{equation}\forall v\in  \tang{x}{X} : Q(v, x- u) = 0 \quad \Longleftrightarrow \quad \forall v'\in  \tang{x'}{X'} : Q'(v', x'- u') = 0,\end{equation}
where $u\in\mathbb R^n$  and $u'=\varphi(u)$. This means that $x\in X$ is a critical point for $u$ if, and only if, $x'$ is a critical point for $u'$. In particular, the number of critical points of $X$ w.r.t.\ to $u$ and of $X'$ w.r.t.~$u'$ is the same. However, since $X$ is assumed to be general, $X'$ is general and therefore $Q'$ and $u'$ can be assumed to be general for $X'$, although we did not define them independently. Consequently, Theorem \ref{main_informal} applies to $X',Q'$ and~$u'$. This gives the number of critical points also for the original setting with $X,Q$ and $u$. The reason why this works is that given a quadric, a general variety always has generic EDD for that quadric. The article \cite{MAXIM2020102101} calls, in the nondegenerate case, the difference between generic EDD and actual EDD the \emph{ED defect}. It would be interesting to study the ED defect also in the setting of degenerate quadrics.

We close this section with discussing the point $u = BA^\dagger$ to which we aim to find a closest point on~$X$. Note that, while the quadric $Q$ in \eqref{special_Q} only depends on the input data in $\mathcal{S}$, the point $u$ depends on both the input and output data. The following result says that, for sufficiently noisy data $\mathcal{S}$, we can assume the point $u$ to be generic, at least after projecting away the kernel of the quadric $Q$ (which is also the only genericity needed for Theorem~\ref{main_informal}).

\begin{prop} \label{prop:largeData}
For fixed dataset size $s$, consider the set of all possible data points $u$, that is,
\begin{equation}U_s := \{ BA^\dagger \mid \mathcal{S} \subseteq \RR^{\numm{in}} \times \RR^{\numm{out}}, |\mathcal{S}| = s  \}.\end{equation}
\begin{enumerate}
    \item If $s \geq \dim \RR[x_1,\ldots,x_{\numm{in}}]_{\leq D}$, then $U_s$ equals the ambient space $(\RR[x_1,\ldots,x_{\numm{in}}]_{\leq D})^{\numm{out}}$.
    \item If $s < \dim \RR[x_1,\ldots,x_{\numm{in}}]_{\leq D}$, then $U_s$ is  the orthogonal complement (w.r.t.\  the standard dot product) of the kernel of the matrix $Q$. 
\end{enumerate}
\end{prop}
\begin{proof}
    The linear span of the Veronese variety $\mathrm{im}(\nu_D)$ is its whole ambient space of dimension equal to $\dim \RR[x_1,\ldots,x_{\numm{in}}]_{\leq D}$.
    Hence, for generic $\mathcal{S}$ with size at least $ \dim \RR[x_1,\ldots,x_{\numm{in}}]_{\leq D}$, the matrix $A$ is of full column rank, and so $A^\dagger$ is of full row rank. Therefore, for varying $B$, we have that $BA^\dagger$ is an arbitrary matrix.

    Now, we assume that  $|\mathcal{S}|$ is less  than $ \dim \RR[x_1,\ldots,x_{\numm{in}}]_{\leq D}$.
    In this case, for varying $B$, 
    the $\numm{out}$ many rows of the matrix $BA^\dagger$ are arbitrary vectors in the row space of $A^\dagger$, which equals the column space of $A$.
    The kernel of the quadric $Q=I_{\numm{out}} \otimes AA^\top$ is $\RR^{\numm{out}}\otimes \ker(A^\top) \cong \ker(A^\top)^{\numm{out}}$, so its orthogonal complement is the $\numm{out}$-fold direct sum of the column space of $A$. Therefore, $BA^\dagger$ is an arbitrary point in the orthogonal complement of $\ker Q$.
\end{proof}

\bigskip
\section{General Varieties}
In this section, we prove the main results of this work. We first show Theorem~\ref{main_informal}, and then count the occurring finite sets of critical points in Section~\ref{ssec:projective}.

Recall that we wish to describe the critical point set $\crit_{X,Q}(u)$ in \eqref{eq:crit}, and from \eqref{eq:crit2} that the criticality condition is $(x-u)\perp_Q \tang{x}{X}$. Analogously to the case of positive definite quadrics studied in \cite{draisma2016euclidean}, a central object in our proofs is the \emph{ED correspondence} 
\begin{equation}\label{def:EXQ}
    E_{X,Q} := 
    \{(x,u) \in \reg{X}(\CC) \times \mathbb{C}^n  \mid (x- u) \perp_Q \tang{x}{X}  \}.
\end{equation} 
By construction, the set of critical points  is the fiber of $u$ under the projection $E_{X,Q} \rightarrow \CC^n$ onto the second factor.

The main difference when compared to the nondegenerate case is that taking orthogonals w.r.t.\ a degenerate quadric $Q$ is more subtle, and exhibits a non-uniform behavior.
Denoting again by $K$ the kernel of the bilinear form associated with $Q$,
 for a linear subspace $L \subseteq \CC^n$, it holds that 
 \begin{equation}\dim  L^{\perp_Q} = n - \dim  L + \dim  L \cap K.\end{equation} This means that the projection $E_{X,Q} \rightarrow \reg{X}$ onto the first factor is not a vector bundle over $\reg{X}$; its fibers are affine subspaces of $\CC^n$ whose dimension depends on how the tangent space $\tang{x}{X}$ intersects~$K$. To partition the ED correspondence accordingly, we define the following subsets of~$X$. Recall from \eqref{def:d} and \eqref{def:k} that we denote $d = \dim  X$ and $k = \dim  K$. Then, we define
 \begin{equation} \label{eq:exactPolarLocus}
    P_{i,K}(X) := \{ x \in \reg{X}(\CC)  \mid\dim  \tang{x}{X} \cap K = i  \},\quad 0 \leq i \leq \textnormal{min}\{ d, k  \}.
\end{equation} 
Setting 
\begin{equation}E_{i,Q}(X) := 
\{(x,u) \in E_{X,Q} \mid x \in P_{i,K}(X)\},\end{equation}
we obtain the partition 
 \begin{equation} \label{eq:decomp_E}E_{X,Q} = \bigcup_i E_{i,Q}(X).
 \end{equation}

The key to proving Theorem~\ref{main_informal} are the following two lemmata. The first lemma implies that each component $E_{i,Q}(X)$  in \eqref{eq:decomp_E} contributes to the critical points of some generic data $u$ only if its dimension is at least $n$. The second lemma then computes the dimension of $E_{i,Q}(X)$ (and also of~$P_{i,K}(X)$). Therefore, these lemmata show us which $E_{i,Q}(X)$ must be studied for determining the number of critical points.

\begin{lem}
    \label{lem:relevantComp}
    If $\dim  E_{i,Q}(X)<n$, then, for generic $u \in \CC^n$,  $\crit_{X,Q}(u) \cap P_{i,K}(X) = \emptyset$.
\end{lem}
\begin{proof}
    As explained above, the critical points form the fiber of $u$ under the projection $E_{X,Q} \rightarrow \CC^n$ onto the second factor.
    Hence, if $\dim  E_{i,Q}(X)<n$, then the fiber over a generic $u \in \CC^n$ misses the component $E_{i,Q}(X)$.
\end{proof}

To compute the dimensions of the components $E_{i,Q}(X)$, we make use of dimension counts in Grassmannians. 
Given a finite-dimensional complex vector space $V$ and $0 \leq m \leq \dim  V$, we denote by $\grass(m, V)$  the \emph{Grassmannian} of $m$-dimensional linear subspaces of $V$. Moreover, we consider the \emph{Gauss map} 
\begin{equation}\label{def_gaussmap}\tau \colon X \dashrightarrow \grass(d, \CC^n), x \mapsto \tang{x}{X};
\end{equation}
i.e., the rational map induced by the tangent bundle of~$\reg{X}$. 

For the next lemma recall again that  $d = \dim  X$ and $k = \dim  K$ (see  \eqref{def:d} and \eqref{def:k}). 
\begin{lem}
    \label{lem:dimensioPolar}
    Let $Q$ be a generic quadric with $k$-dimensional kernel. 
    If both $i \geq d+k-n$ and $\dim  \tau(X) \geq i(i+n-d-k)$, we have 
    \begin{align}\label{eq:dimpi}
        \dim  P_{i,K}(X) &= d - i(i+n-d-k) \quad \text{ and}\\
    \label{eq:dimei}
        \dim  E_{i,Q}(X) &= n-i(i+n-d-k-1).
    \end{align}
    In all other cases, we have $P_{i,K}(X) = \emptyset$ and $ E_{i,Q}(X)=\emptyset$.
\end{lem}
\begin{proof}
For all $x \in \reg{X}$, we have that $\dim  \tang{x}{X} \cap K \geq d+k-n$. Thus, $P_{i,K}(X) = \emptyset$ whenever $i < d+k-n$.
From now on, we fix $i \geq d+k-n$.
Consider the subspace of the Grassmannian 
\begin{equation} 
        Z_{i,K} := \{ T \in \grass(d,\CC^n) \mid \dim  T \cap K = i \}.
  \end{equation}
Note that $P_{i,K}(X) = \tau^{-1}(Z_{i,K})$. 
We compute the dimension of $P_{i,K}(X)$ via the intersection of $Z_{i,K}$ with the tangent bundle $\tau(X)$, as follows.

We begin by computing the dimension of $Z_{i,K}$. Its Zariski closure is  
\begin{equation}
\begin{aligned}
              \{ T \in \grass(d,\CC^n) \mid \dim  T \cap K \geq i \} 
            &= \displaystyle \bigcup_{L \in \grass(i, K)} \underbrace{\left\{  T \in \grass(d,\CC^n) \mid L \subseteq T   \right\}}_{\simeq  \grass(d- i, \CC^{n - i})}. 
\end{aligned}
\end{equation}
Therefore, $\dim  Z_{i,K} = \dim   \grass(i, K) + \dim   \grass(d- i, \CC^{n - i}) = i(k-i) + (d-i)(n-d)$. As $K$ varies, $Z_{i,K}$ gets translated via the action by $\textnormal{GL}(\CC^n)$ on $\grass(d, \CC^n)$. Hence, by Kleiman's transversality theorem \cite{kleiman1974transversality}, for generic $K$, we have that the intersection of $Z_{i,K}$ with $\tau(X)$ has the expected dimension: 
\begin{equation} \label{eq:dimTauZ}
\dim  \tau(X) \cap Z_{i,K} = \dim  Z_{i,K} + \dim  \tau(X) - \dim  \grass(d, \CC^n) =  \dim  \tau(X) - i(i+n-d-k).
\end{equation}
This number being negative means that the intersection is empty, in which case also $P_{i,K}(X)$ is empty. From now on, we assume that \eqref{eq:dimTauZ} is non-negative so that the intersection is not empty. 
To compute
the dimension of  $P_{i,K}(X)$, we consider the linear spaces that are tangent to $X$ at a larger-dimensional locus than expected:
\begin{equation} 
    \Delta := \{ T \in \tau(X)  \mid  \dim  \tau^{-1}(T) > d - \dim  \tau(X) \}.
\end{equation}
The space $\Delta$ is a lower-dimensional Zariski closed subset of $\tau(X)$.
Hence, again by Kleiman's transversality, we have  $\dim  \Delta \cap Z_{i,K} < \dim  \tau(X) \cap Z_{i,K}$ for generic $K$. 
The latter means that a generic linear space in the intersection $\tau(X) \cap Z_{i,K}$ is tangent to $X$ at a sublocus of dimension $d-\dim  \tau(X)$, showing that 
$\dim  P_{i,K}(X) = \dim  (\tau(X) \cap Z_ {i,K}) + (d - \dim  \tau(X))$.
Together with~\eqref{eq:dimTauZ}, this implies  \eqref{eq:dimpi}. 

Lastly, for all $x \in P_{i,K}(X)$, we have that $\dim   \tang{x}{X}^{\perp_Q} = n - d + i$. This implies that we have $ \dim  E_{i,Q}(X)  =   \dim  P_{i,K}(X) + n - d + i$ and with this \eqref{eq:dimei} follows. 
\end{proof}

We now prove Theorem~\ref{main_informal} by investigating the components $E_i(X,Q)$ of dimension at least $n$ and their behavior under the  orthogonal projection $\pi: \RR^n \to K^\perp$ onto the orthogonal complement of~$K$.
We treat the  two cases in Theorem~\ref{main_informal} separately.

\subsection{Mildly degenerate quadrics ($k+d < n$)}
\label{sec:mildlyDegenerate}

Throughout this section, we assume that $k+d < n $, meaning that $ \dim K + \dim X < n$.

We denote the Zariski closure of the projection of $X$ by 
\begin{equation}\label{def:Y}Y := \textnormal{cl}\,\pi(X).\end{equation}
In the setting of this section, $Y$ is a proper subvariety of~$K^\perp$. Moreover, for generic $K$, we have that $\tang{x}{X} \cap K = \{ 0 \}$ for generic $x \in \reg{X}$. In other words, $P_{0,K}(X)$ is Zariski dense in $X$. 
In fact, none of the other loci $P_{i,K}(X)$ for $i>0$ is expected to contain critical points, which follows from Lemma~\ref{lem:relevantComp} in combination with the next statement.

\begin{lem}\label{lem:dimEmild}
For a generic quadric $Q$ with $k$-dimensional kernel, it holds that:
\begin{equation}\label{eq:case1}
    \dim  E_{i,Q}(X) \; \begin{cases}
    = n & \textnormal{ if } i=0, \\
    < n & \textnormal{ if } i > 0.
\end{cases}
\end{equation}
\end{lem}
\begin{proof}
    Since $n-d-k > 0$, \eqref{eq:case1} follows from Lemma~\ref{lem:dimensioPolar}.
\end{proof}

Now, we finally prove the first case of our Main Theorem \ref{main_informal}. 
For that, recall that $\pi$ turns the degenerate quadric $Q$ into a nondegenerate quadratic form $\pi(Q)$ on $K^\perp$. 
\begin{thm}\label{prop:generalXsmallK}
As in \eqref{def:Y}
denote $Y= \textnormal{cl}\,\pi(X)$.
For a generic quadric $Q$  with a $k$-dimensional kernel and a generic $u \in \RR^n$, the projection $\pi$ induces a bijection between the critical point sets
\begin{equation}
    \crit_{X,Q}(u) \overset{1:1}{\longleftrightarrow}
    \crit_{Y,\pi(Q)}(\pi(u)).
\end{equation}
In particular, $|\crit_{X,Q}(u)|= \textnormal{gEDD} (Y)$. 
\end{thm}
\begin{proof}
Recall that $\crit_{X,Q}(u)$ is the fiber over $u$ of the projection $E_{X,Q} \to \CC^n$ onto the second factor.
By Lemmata~\ref{lem:relevantComp} and \ref{lem:dimEmild}, 
that fiber is the same as the fiber of the restricted projection 
$E_{0,Q}(X) \to \CC^n$, and 
 $\crit_{X,Q}(u) \subseteq P_{0,K}(X)$, where $K$ is the kernel of $Q$.

Now, we consider the orthogonal projection $\pi\colon \RR^n \to K^\perp$.
Due to the genericity of $K$ and the assumption in this section that $\dim  K = k < n-d $,
the restricted projection $\pi|_X$ is birational onto its image, i.e., for almost all $x \in X$, we have that $\pi^{-1}(\pi(x)) = \{ x \}$.
Hence, \begin{equation}\Delta := \{ x \in X \mid |\pi^{-1}(\pi(x)) \cap X| > 1 \}\end{equation}
is a lower-dimensional subset of $X$.
Therefore, 
\begin{equation}\label{eq:dimDelta}
   \dim \{ (x,v) \in E_{0,Q}(X) \mid x \in \Delta \} = \dim (\Delta \cap P_{0,K}(X)) + n - d < n,
\end{equation}
and so the genericity of $u$ implies that $\crit_{X,Q}(u) \cap \Delta = \emptyset$.

All in all, $\crit_{X,Q}(u)$ is contained in the locus $P_{0,K}(X) \setminus \Delta$.
That locus contains only smooth points of $X$, and its Zariski closure is all of $X$.
The projection $\pi$ restricted to that locus is injective with injective differentials. Thus, for all $x$ in that locus, we have that $\pi(x)$ is a smooth point of $Y$ and $\pi (\tang{x}{X}) = \tang{\pi(x)}{Y}$.
Therefore, every $x\in \crit_{X,Q}(u)$ (which satisfies $(x-u)\perp_Q\tang{x}{X}$) also satisfies $(\pi(x)-\pi(u)) \perp_{\pi(Q)} \tang{\pi(x)}{Y}$,  i.e., $\pi(x) \in \crit_{Y,\pi(Q)}(\pi(u))$.

For the converse direction of the proof, we make a similar argument for the ED correspondence of~$Y$ by considering the set 
\begin{equation}
    \Delta' := \{ y \in \reg{Y} \mid y \notin \pi(X) \quad \text{ or } \quad |\pi^{-1}(y) \cap X| > 1 \quad \text{ or } \quad \pi^{-1}(y) \cap X \not\subseteq P_{0,K}(X) \}. 
\end{equation}
Then, $\Delta'$ is a lower-dimensional subset of $Y$. As with \eqref{eq:dimDelta}, we conclude that the dimension of 
$ \{(y,v) \in  E_{Y,\pi(Q)} \mid y \in \Delta' \}$ is less than $\dim K^\perp$. Therefore, since $\pi(u)$ is a generic point in~$K^\perp$, we have 
$\crit_{Y,\pi(Q)}(\pi(u)) \cap \Delta' = \emptyset$.
Hence, for every $y \in \crit_{Y,\pi(Q)}(\pi(u))$, there is a unique $x \in X$ with $\pi(x)=y$; and moreover, $x \in P_{0,K}(X)$. This implies again $\tang{y}{Y} = \pi(\tang{x}{X})$.
Thus, $(y-\pi(u)) \perp_{\pi(Q)} \tang{y}{Y}$ implies $(x-u) \perp_Q \tang{x}{X}$. Consequently, for every $y \in \crit_{Y,\pi(Q)}(\pi(u))$ there is a unique $x \in \crit_{X,Q}(u)$.
\end{proof}

\subsection{Significantly degenerate quadrics ($k+d \geq n $)} \label{sec:veryDegenerate}
In this section, we assume $k+d \geq n $, which means $\dim K + \dim X \geq n$.

In that case, $Y=\textnormal{cl}\,\pi(X)$ for generic $K$ is all of $K^\perp$.
Moreover, a generic tangent space of $X$ and the kernel  $K$ span the whole ambient space $\mathbb{C}^n$.
The expected dimension of the intersection between a tangent space and $K$ is $d+k-n$, meaning that $P_{d + k - n,K}(X)$ is Zariski dense in $X$. The next statement, together with Lemma~\ref{lem:relevantComp}, shows that we also need to consider points where the intersection between $\tang{x}{X}$ and $K$ is $d+k-n+1$; i.e., one higher than expected. This means that, in general, also $P_{d + k - n +1,K}(X)$ contributes to the critical points for a generic data point $u$. 

Recall from \eqref{def:d} and \eqref{def:k} that $d=\dim X$ and $k=\dim K$.
\begin{lem}\label{lem:dimEsignificant}
For a generic quadric $Q$ with $k$-dimensional kernel, it holds that:
    \begin{equation}\label{eq:case2}
    \dim  E_{i,Q}(X) \; \begin{cases}
    = d + k & \textnormal{ if } i=d + k - n, \\
    = n & \textnormal{ if } i = d+ k -n +1  \textnormal{ and } d-\dim  \tau(X) < n-k, \\
    < n & \textnormal{ otherwise. }
\end{cases}
\end{equation}
\end{lem}
\begin{proof}
        Since $n-d-k \leq 0$, \eqref{eq:case2} follows from Lemma~\ref{lem:dimensioPolar}.
\end{proof}

A general variety $X$ of dimension $d=\dim X$ satisfies that $\dim \tau(X) = d$. In that case, the second condition from \eqref{eq:case2} is always satisfied (as long as the quadric $Q$ is non-zero). Hence, we typically expect both loci $P_{d + k - n,K}(X)$ and $P_{d + k - n+1,K}(X)$ to contribute to the critical point set.

The Zariski closure of $P_{d + k - n+1,K}(X)$ is the \emph{ramification locus} of the orthogonal projection $\pi$ restricted to the variety $X$:
\begin{equation}  \label{eq:Ram}
\begin{aligned}
 \mathrm{Ram}(\pi |_X) :=& \textnormal{cl} \ P_{d + k - n+1,K}(X) \\
 =& \textnormal{cl} \  \{ x \in \reg{X} \mid K+x \textnormal{ intersects } X \textnormal{ at } x \textnormal{ non-transversely} \}. 
 \end{aligned}
\end{equation}
In other words, $ \mathrm{Ram}(\pi |_X) $ is the Zariski closure of the set of critical points of $\pi$ over $X$ (this has been informally introduced in Section~\ref{sec:overviewResults} without taking the Zariski closure).
The \emph{branch locus} \begin{equation}\mathrm{Br}(\pi |_X) := \textnormal{cl}\, \pi(\mathrm{Ram}(\pi |_X))\end{equation} is the Zariski closure of the corresponding set of critical values.
By Lemma~\ref{lem:dimensioPolar}, 
\begin{equation}\dim \mathrm{Ram}(\pi |_X)  = n-k-1,\end{equation}
and so the branch locus is expected to be a hypersurface in $K^\perp$.
In fact, we will prove that the finitely many critical points on that hypersurface with respect to the nondegenerate quadric $\pi(Q)$ are expected to be in bijection with the finitely many critical points described in the second case of our Main Theorem~\ref{main_informal}.
For that, we need the following statement. 

\begin{lem}\label{lem:piBirationalOnRam}
    For a sufficiently general variety $X$, the projection $\pi$ is birational over the ramification locus $\mathrm{Ram}(\pi |_X)$, i.e., it is injective almost everywhere. 
\end{lem}

We postpone the proof of that lemma to the end of this section, where we also specify concrete assumptions on $X$ that make it sufficiently general.
Now, we provide the remaining parts of the proof of our Main Theorem~\ref{main_informal}.

\begin{thm} \label{prop:generalXlargeK}
For a generic quadric $Q$  with a $k$-dimensional kernel and a generic $u \in \RR^n$, the critical point set
$\crit_{X,Q}(u)$ consists of the zero-loss solutions  $(K+u) \cap \reg{X}$, plus at most finitely many points on the ramification locus $\mathrm{Ram}(\pi |_X)$.
If $X$ is sufficiently general,
those finitely many points are in bijection, via $\pi$, with
$\crit_{\mathrm{Br}(\pi |_X), \pi(Q)}(\pi(u))$;
in particular, their cardinality is~$\textnormal{gEDD} (\mathrm{Br}(\pi |_X))$.
\end{thm}

\begin{proof}
Let $\delta := d+k-n$.
By Lemmata~\ref{lem:dimEsignificant} and~\ref{lem:relevantComp}, we have that $\crit_{X,Q}(u)$ is the union of the fibers over $u$ of the projections 
$E_{\delta, Q}(X) \to  \CC^n$ 
and $E_{\delta+1, Q}(X) \to  \CC^n$ onto the second factor.

First, we consider $E_{\delta, Q}(X) \to  \CC^n$. We show that its fiber $\mathcal{F}_\delta(u)$ over $u$ equals $(K+u) \cap \reg{X}$. We begin by showing that the first set is a subset of the latter. 
By definition, $\mathcal{F}_\delta(u)$ is contained in $P_{\delta,K}(X) \subseteq \reg{X}$.
For all $x \in P_{\delta,K}(X)$, we have that $(\tang{x}{X})^{\perp_Q} = K$.
Therefore, $\mathcal{F}_\delta(u)$ is indeed contained in $\reg{X} \cap (K+u)$. 
For the converse direction, we recall from Lemma~\ref{lem:dimensioPolar} that the dimension of the ramification locus $\mathrm{Ram}(\pi |_X)=\textnormal{cl}(P_{\delta+1,K}(X))$ is $n-k-1$ and thus the genericity of $u$ implies that $K+u$ does not intersect $\mathrm{Ram}(\pi |_X)$.
In other words, $(K+u) \cap \reg{X} \subseteq P_{\delta,K}(X)$.
As argued above, all points $x$ in that intersection satisfy $(x-u) \perp_Q \tang{x}{X}$, and so we conclude that~$(K+u) \cap \reg{X} \subseteq \mathcal{F}_\delta(u)$. 

Second, we consider $E_{\delta+1, Q}(X) \to  \CC^n$ and its fiber $\mathcal{F}_{\delta+1}(u)$ over a generic $u$.
Since $\dim  E_{\delta+1, Q}(X) \leq n$ by Lemma~\ref{lem:dimEsignificant}, the fiber $\mathcal{F}_{\delta+1}(u)$ consist of at most finitely many points. As argued in \eqref{eq:Ram}, those lie on the ramification locus $\mathrm{Ram}(\pi |_X)$.

For the rest of this proof, we assume that $X$ is a sufficiently general variety.
In particular, $\dim  \tau(X) = d$ and so $\dim  E_{\delta+1, Q}(X) = n$ by Lemma~\ref{lem:dimEsignificant}.
We show that the fiber $\mathcal{F}_{\delta+1}(u)$
is in bijection with $\crit_{\mathrm{Br}(\pi |_X), \pi(Q)}(\pi(u))$ via the projection $\pi$.
For that, we proceed analogously to the proof of Theorem~\ref{prop:generalXsmallK}.

By Lemma~\ref{lem:piBirationalOnRam}, the projection $\pi$ induces a birational map $\mathrm{Ram}(\pi |_X) \to \mathrm{Br}(\pi |_X)$.
To simplify notation, we write $R:= \mathrm{Ram}(\pi |_X)$ and $B := \mathrm{Br}(\pi |_X)$. We will first show that this map takes critical points in $P_{\delta+1,K}(X)$ to critical points in $\mathrm{Br}(\pi|_X)$.
Define the subset 
\begin{equation}
    \Delta := \Big\{ x \in P_{\delta+1,K}(X) \;\Big|\;  |\pi^{-1}(\pi(x)) \cap  R|>1 \quad \text{ or } \quad  x \notin \reg{R}\Big\},
\end{equation}
which we need to show our critical points generically avoid. This is a lower-dimensional subset of  $P_{\delta+1,K}(X)$.
Therefore, 
$\{(x,v) \in E_{\delta+1,Q}(X) \mid x\in \Delta \}$
has dimension 
$\dim (\Delta) + n-d+\delta+1 < n$, and so the genericity of $u$ implies that $\mathcal{F}_{\delta+1}(u) \subseteq P_{\delta+1,K}(X) \setminus \Delta$.
The latter locus contains only smooth points of $R$, and its Zariski closure is all of $R$.
The projection $\pi$ restricted to that locus is injective with injective differentials. 
Thus, for all $x$ in that locus, $\pi(x)$ is a smooth point of $B$ and~$\pi(\tang{x}{R}) = \tang{\pi(x)}{B}$.
Moreover, since 
 $\dim \pi(\tang{x}{X}) = d-(\delta+1) = n-k-1 = \dim B$, we even have that $\pi(\tang{x}{X}) = \tang{\pi(x)}{B}$.
Therefore, every $x\in \mathcal{F}_{\delta+1}(u)$ (which satisfies $(x-u)\perp_Q\tang{x}{X}$) also satisfies $(\pi(x)-\pi(u)) \perp_{\pi(Q)} \tang{\pi(x)}{B}$,  i.e., $\pi(x) \in \crit_{B,\pi(Q)}(\pi(u))$.

For the converse direction, we consider
\begin{equation}
    \Delta' := \{ y \in \reg{B} \mid y  \notin \pi(R) \quad \text{ or } \quad |\pi^{-1}(y) \cap R| > 1 \quad \text{ or } \quad \pi^{-1}(y) \cap R \not\subseteq P_{\delta+1,K}(X) \}.
\end{equation}
This is a lower-dimensional subset of $B$, and so 
$\dim \{ (y,v) \in E_{B,\pi(Q)} \mid y \in \Delta' \} < \dim K^\perp$.
Since $\pi(u)$ is a generic point in $K^\perp$, 
$\crit_{B,\pi(Q)}(\pi(u)) \cap \Delta' = \emptyset$.
Hence, for every $y \in \crit_{B,\pi(Q)}(\pi(u))$, there is a unique $x \in R$ with $\pi(x)=y$. Moreover, we have $x \in P_{\delta+1,K}(X)$, which implies again $\tang{y}{B} = \pi(\tang{x}{X})$.
Thus, $(y-\pi(u)) \perp_{\pi(Q)} \tang{y}{B}$ implies $(x-u) \perp_Q \tang{x}{X}$, i.e., $x \in \mathcal{F}_{\delta+1}(u)$.
\end{proof}

\paragraph{Proof of Lemma~\ref{lem:piBirationalOnRam}}
We provide a proof based on projective geometry and, in particular, projective duality. 

The $n$-dimensional complex projective space $\PP^n$ is 
\begin{equation}\label{def_PP}
\PP^n := (\mathbb{C}^{n+1} \setminus \lbrace 0 \rbrace) / \sim,
\end{equation} where two non-zero vectors $v,w$ are identified (i.e., $v \sim w$) if, and only if, they are proportional (i.e., $w = \lambda v$ for some $\lambda \in \mathbb C$).
The \emph{dual projective space} $(\PP^n)^\ast$ is constructed in the same way from the dual vector space $(\mathbb C^{n+1})^\ast$. In other words, $(\PP^n)^\ast$ is the space of all hyperplanes in $\PP^n$.
A \emph{projective variety} $X \subseteq \PP^{n}$ is the solution set of a system of homogeneous polynomial equations in $n+1$ variables.
The \emph{dual variety} $X^\vee \subseteq (\PP^n)^\ast$ is the Zariski closure of the set of all hyperplanes tangent to $\reg{X}$.
For instance, if $X$ is a point, then $X^\vee$ is a hyperplane in $(\PP^n)^\ast$. 
Slightly more generally, if $X$ is a projective subspace of dimension $d$, then $X^\vee$ is a projective subspace of $(\PP^n)^\ast$ of dimension $n-d-1$, which satisfies the biduality $(X^\vee)^\vee = X$.
Over the complex numbers, also nonlinear projective varieties enjoy the property of biduality.

\begin{fact}[{\cite[Ch. 1, Thm. 1.1]{gkz}}]
    Let $X \subseteq \PP^{n}$ be a projective variety. Moreover,
    let $x \in \reg{X}$ and $H \in \reg{X^\vee}$. 
    The hyperplane $H$ is tangent to $X$ at $x$ if, and only if, the hyperplane $x^\vee$ is tangent to $X^\vee$ at the point $H^\vee$.
    In particular, we have $(X^\vee)^\vee = X$.
\end{fact}

To make use of projective duality, we proceed as follows.
We consider $\mathbb{C}^n$ as an affine chart of $\PP^n$, by identifying each vector $(x_1,\ldots,x_n)$ with the equivalence class of $(x_1,\ldots,x_n,1)$ under $\sim$.
Then, starting from a real variety $X \subseteq \RR^n$, we consider the Zariski closure $\bar X$ of $X(\CC) \subseteq \CC^n \subseteq \PP^n$ inside the complex projective space $\PP^n$.
The dimension of $X$ inside $\PP^n$ is the same as the real dimension of $X$, which we denote by $d$, as before.
The orthogonal projection $\pi: \RR^n \to K^\perp$ with kernel $K$ of dimension $k$ corresponds in the projective setting to a projection 
$\bar \pi \colon \PP^n \dashrightarrow \PP^{n-k}$ that is undefined precisely at the $(k-1)$-dimensional subspace $\bar K$ that is the intersection of the Zariski closure of $K$ inside $\PP^n$ with the hyperplane $H$ at infinity (i.e., the hyperplane $ H \subseteq \PP^n$ where the last coordinate is zero).
We will now prove, for sufficiently general $X$, that the ramification and branch loci of projecting $\bar X$ via $\bar \pi$ are indeed birational to each other via $\bar  \pi$.
For $X$ to be `sufficiently general', it is enough to assume that it satisfies the following three properties: 1) it is smooth, 2) its closure $\bar X$ meets the hyperplane $H$ at infinity transversely, and 3) the dual variety $\bar X^\vee$ is a hypersurface (cf. \cite[Ch. 1, Cor. 1.2]{gkz}). 

\begin{lem} \label{lem:projBirational}
Suppose that $H$ intersects $\bar X$ only at smooth points, and that the intersection is transversal therein. 
Then, for a generic subspace $\bar K \subseteq H$ of dimension $k-1$, $\bar \pi$ has finite fibers on its ramification locus $\mathrm{Ram}(\bar \pi |_{\bar X})$. Suppose moreover that $\bar X^\vee$ is a hypersurface. Then, for a generic subspace $\bar K \subseteq H$ of dimension $k-1$, $\bar \pi$ is birational on its ramification locus $\mathrm{Ram}(\bar \pi |_{\bar X})$, i.e., it is bijective to the branch locus $\mathrm{Br}(\bar \pi |_{\bar X})$ almost everywhere. 
\end{lem}

\begin{proof}
To simplify notation, we write $X, \pi, K$ instead of $\bar X, \bar \pi, \bar K$.

Since $H$ intersects $X$ transversally (at smooth points), the classical theorem of Bertini \cite{kleiman1997bertini} implies that a generic $K \subset H$ will intersect $X$ only at smooth points and transversally as well. Thus, for generic $K$, we have that $K \cap \mathrm{Ram}(\pi |_X) = \emptyset$. This implies that the projection $\pi$ restricted to $\mathrm{Ram}(\pi |_X) $  has finite fibers (indeed, if a fiber $\pi^{-1}(y) \cap \mathrm{Ram}(\pi |_X)$ would contain a curve, then -- since $K$ is a hyperplane in the projective space $\pi^{-1}(y)$ -- that curve would intersect $K$, which contradicts that $K \cap \mathrm{Ram}(\pi |_X) = \emptyset$).

We now show the second claim. For $x \in P_{d + k - n+1,K}(X) $, we have that $H_x := K + \tang{x}{X} $ is a hyperplane in $\PP^n$, since $\dim  H_x = \dim  \tang{x}{X} + \dim  K - (d + k - n+1)  = n - 1$. In particular, $H_x^\vee$ belongs to $X^\vee \cap K^\vee$. Hence, we obtain a rational map \begin{equation}\alpha \colon \mathrm{Ram}(\pi |_X) = \textnormal{cl} \ P_{d + k - n+1,K}(X) \dashrightarrow  X^\vee \cap K^\vee\end{equation} defined over $P_{d + k - n+1,K}(X)$ as $x \mapsto H_x^\vee$. Since $X^\vee$ is a hypersurface, this map is birational. Indeed, by biduality, a rational inverse is provided by the Gauss map \eqref{def_gaussmap} $X^\vee \dashrightarrow (X^\vee)^{\vee} = X$ restricted to $X^\vee \cap K^\vee$. That restriction is a well-defined rational map, since -- due to the genericity of $K$ -- $X^\vee \cap K^\vee$ is not contained in the singular locus of $X^\vee$. 

Note that for a generic $x \in \mathrm{Ram}(\pi |_X)$, we have by biduality that \begin{equation}\tang{H_x^\vee}{(X^\vee \cap K^\vee)} = \tang{H_x^\vee}{X^\vee \cap K^\vee} = x^\vee \cap K^\vee.\end{equation} 
The intersection $x^\vee \cap K^\vee$ depends only on $\pi(x)$ (i.e., not on $x$ itself). 
Hence, the following is a well-defined map over the projectivized quotient $\PP^{n-k}$: $\beta(y) := x^\vee \cap K^\vee$ for some $x \in \pi^{-1}(y)$. Putting everything together,  we obtain the following commutative diagram:
\[\begin{tikzcd}
	{\mathrm{Ram}(\pi |_X) } & {X^\vee \cap K^\vee} & {\grass(n-k-1, (\PP^n)^\ast)} \\
	{\PP^{n-k}}
	\arrow["\alpha", dashed, from=1-1, to=1-2]
	\arrow["\pi"', from=1-1, to=2-1]
	\arrow["\delta", dashed, from=1-2, to=1-3]
	\arrow["\beta", from=2-1, to=1-3]
\end{tikzcd}\]
Here, $\delta$ is the Gauss map \eqref{def_gaussmap} of $X^\vee \cap K^\vee$. Now, $x^\vee \cap K^\vee = {\tilde x}^\vee \cap K^\vee$ if, and only if, $\pi(x) = \pi({\tilde x})$. In other words, $\beta$ is  injective. 
Since $\pi$ has finite fibers by the first part of the proof and $\beta$ is injective, the generic fibers of $\delta$ must be finite as well. By \cite[Theorem 2.3]{zak1993tangents}, the generic fibers of projective Gauss maps are projective subspaces, and so we conclude that the generic fiber of the Gauss map $\delta$ must be a point. In other words, $\delta$ is birational onto its image. Since $\alpha$ is also birational, we have 
\begin{equation}\pi^{-1} = \alpha^{-1}\circ \delta^{-1}\circ \beta.\end{equation}
That means $\pi$ is birational as well.    
\end{proof}

\subsection{Counting over Projective Varieties} \label{ssec:projective}
In this section, we consider varieties $X \subseteq \RR^n$ that are \emph{affine cones} over projective varieties; that is, $X$ is the solution set in $\mathbb{R}^n$ of \emph{homogeneous} polynomials in $n$ variables. For this, we denote real projective space by $\PP_\RR^n$ (defined by replacing $\CC$ by $\RR$ in~\eqref{def_PP}).

Many neuromanifolds are such affine cones.
For instance, if the activation $\sigma$ of the multilayer perceptron in Example~\ref{exmp:determinantal} is a monomial $\sigma(x) = x^r$, then its neuromanifold is an affine cone.
The same holds for similar architectures with weight-sharing restrictions, such as convolutional neural networks \cite{shahverdi2024geometryoptimizationpolynomialconvolutional}. 
Also, the neuromanifolds of deep networks where each layer is a lightning self-attention mechanism \eqref{eq:attention} are affine cones \cite{henry2024geometry}.

All results from the sections above apply to affine cones over projective varieties.
In particular, for a sufficiently general projective variety, its affine cone is general enough for Lemma~\ref{lem:piBirationalOnRam} and the full version of Theorem~\ref{prop:generalXlargeK} to hold. 
Indeed, if $X \subseteq \RR^n$ is the affine cone over a general projective variety $Z\subset \PP^{n-1}_\RR$ and $\bar X$ is the Zariski closure of $X$ in complex projective space $\PP^n$, then Lemma~\ref{lem:projBirational} applies due to the following reasons: for a general $Z$, we have that 1)  $Z$ is smooth and so is $\bar X$, 2) the hyperplane $H$ at infinity intersects $\bar X$ transversely, and 3) $Z^\vee$ is a hypersurface and thus $\bar X^\vee$ is a hypersurface as~well.

A crucial advantage of the projective scenario is that it enables an elegant counting of critical points via characteristic classes.
To this end, we recall two classical concepts from projective geometry: 1)~degrees and 2) polar varieties.

The \emph{degree} of a projective variety $Z\subset \PP^{n-1}_\CC$ is its number of intersection points with a generic projective subspace $L \subseteq \PP^{n-1}_\CC$ of dimension $n-1-\dim(Z)$. It is denoted by $\deg(Z)$. For an affine cone $X \subseteq \mathbb{C}^{n}$ over a projective variety $Z$, we set $\deg(X) := \deg(Z)$. 

Given an affine cone $X \subseteq \mathbb{C}^{n}$ over a projective variety and a linear subspace $V \subseteq \mathbb{C}^{n}$, the corresponding \emph{polar variety} is
\begin{align} \label{eq:polarVariety}
    \mathcal{P}(X,V) := \textnormal{cl} \ \{ x \in \reg{X} \setminus V \mid V+x \text{ intersects } X \text{ at } x \text{ non-transversely} \}.
\end{align}
Note that, in the setting of Section~\ref{sec:veryDegenerate}, if $X$ is an affine cone, then $\mathrm{Ram}(\pi |_X) = \mathcal{P}(X,K)$; see \eqref{eq:Ram}.
More generally, the non-transversality condition in \eqref{eq:polarVariety} (which means that $\dim (V + \tang{x}{X}) < n$) is equivalent to $\dim (V \cap \tang{x}{X}) \geq \dim  V +d -n+1$, where $d$ is the dimension of the affine cone $X$.
Thus, for $\dim  V \geq n-d-1 $, we have that 
\begin{align} \label{eq:polarOurVsClassical}
    \mathcal{P}(X,V) = \textnormal{cl} \ P_{\dim  V +d -n+1, V} (X);
\end{align}
cf.\ \eqref{eq:exactPolarLocus}.
For generic linear spaces $V$ of fixed dimension $j$, the degrees of the polar varieties $\mathcal{P}(X,V)$ are the same \cite{piene1978polar}.  It is thus possible to define the \emph{$j$-th polar degree of $X$} as
\begin{equation}
    \delta_j(X) := \degg \mathcal{P}(X, V), 
\end{equation}
where $V \subseteq \CC^n$ is a generic linear subspace of dimension $j$.
 When $X$ is smooth, its polar degrees are equivalent to the Chern numbers of its tangent bundle; in fact, they can be computed from one another via a recursive formula \cite{holme1988geometric}.

A fundamental result for projective varieties states that the generic Euclidean distance degree of the associated affine cone coincides with the sum of all the polar degrees \cite[Theorem 5.4]{draisma2016euclidean}\footnote{There is no standard convention on how to index the polar degrees. Here, we adhere to the opposite order of the indexing convention in \cite{draisma2016euclidean}, which is why the indices in our sum are different (cf. \cite[Thm. 4.16]{mag}).}:
\begin{align} \label{eq:polarFormula}
    \textnormal{gEDD}(X) = \sum_{j=n-d-1}^{n-1} \delta_j(X).
\end{align}

Since polar degrees stay invariant under generic projections \cite{piene1978polar},  
an interesting consequence of \eqref{eq:polarFormula} is the following:
\begin{fact}[{\cite[Cor. 6.1]{draisma2016euclidean}}]
\label{fact:projection}
    Let $X \subseteq \mathbb{R}^n$ be the affine cone over a projective variety, $K \subseteq \mathbb{R}^n$ be a generic linear subspace of dimension $\dim K =k< n-d = n-\dim  X$, and as in \eqref{def:Y} denote $Y=\textnormal{cl} \ \pi(X)$. Then, the orthogonal projection $\pi: \mathbb{R}^n \to K^\perp$ satisfies that
    \begin{align}
        \textnormal{gEDD}(Y) = \textnormal{gEDD} (X).
    \end{align}
\end{fact}
In particular, for an affine cone $X$ over a projective variety, the number of critical points in Theorem~\ref{prop:generalXsmallK} is simply the generic Euclidean distance degree of $X$ itself.

We believe that this fact holds true for all varieties (i.e., not just for affine cones), but we are not aware of a proof at the moment. Proving the following would be a fundamental contribution to the theory of Euclidean distance degrees.
\begin{conj}
    Fact~\ref{fact:projection} holds for arbitrary varieties.
\end{conj}

Now, we turn our attention to the setting of Section~\ref{sec:veryDegenerate} and the computation of $\textnormal{gEDD} (\mathrm{Br}(\pi |_X))$.
Recall that the finitely many critical points on the ramification locus in Theorem~\ref{prop:generalXlargeK} appear as the generic fiber of the projection 
$E_{d+k-n+1,Q}(X) \to \CC^n$ onto the second factor.
If $X$ is the affine cone over a general projective variety and has dimension $d$, then $\dim  \tau(X) = d-1$.
Thus, in the case that $K$ is a hyperplane (i.e., $k = n-1$), we have that $E_{d+k-n+1,Q}(X) = \emptyset$ by Lemma~\ref{lem:dimensioPolar}, meaning that there are no critical points on the ramification locus in addition to the zero-loss solutions; see Example~\ref{ex:affineConeSurface} below.
Whenever $k < n-1$ (while still assuming $k \geq n-d$), we have that $\dim E_{d+k-n+1,Q}(X) = n$ by Lemma~\ref{lem:dimEsignificant}, and so -- by Theorem~\ref{prop:generalXlargeK} -- we expect finitely many critical points on  $\mathrm{Ram}(\pi |_X)$
of cardinality $\textnormal{gEDD} (\mathrm{Br}(\pi |_X))$.

\begin{exmp}\label{ex:affineConeSurface}[$n=3, d=2, k=2$.]
    For a general surface $X$ in $\RR^3$, its set of tangent planes $\tau(X)$ exhausts (almost) all origin-passing planes in $\RR^3$. Hence, we expect a given plane $K$ to coincide with the tangent plane at finitely many points; see also Figure~\ref{fig:ellPlane} and Lemma~\ref{lem:dimEsignificant}, which yields that $\dim  P_{2,K}(X)=0$.
    However, for an affine cone  $X$ over a projective plane curve, the set of tangent planes $\tau(X)$ is one-dimensional and so we expect none of those planes to coincide with $K$, meaning that  $P_{2,K}(X) = \emptyset$.
    \hfill $\diamondsuit$
\end{exmp}

For affine cones over projective varieties, we can generalize the formula \eqref{eq:polarFormula} to a count of $\textnormal{gEDD} (\mathrm{Br}(\pi |_X))$ by truncating the sum of polar classes:

\begin{thm}\label{thm:polardegs}
    Let $X \subseteq \RR^n$ be the affine cone over a projective variety. Suppose that $X$ satisfies the assumptions of Lemma~\ref{lem:projBirational}. Then, a generic subspace $K \subseteq \RR^n$ of dimension $k \geq n - \dim  X$ satisfies
    \begin{align}
        \textnormal{gEDD} (\mathrm{Br}(\pi |_X)) = \sum_{j=k}^{n-1} \delta_{j} (X).
    \end{align}
\end{thm}
To prove this claim, we first show the following technical result on polar varieties. 

\begin{lem}\label{lemm:polarrecurs}
Let $X \subseteq \RR^n$ be the affine cone over a projective variety, and let $K \subseteq \mathbb{R}^n$
be a linear subspace of dimension $k \geq n - \dim  X$. Suppose that the orthogonal projection $\pi \colon \mathbb{R}^n \rightarrow  K^\perp$ satisfies  $\dim  \pi (\mathcal{P}(X,K)) = \dim   \mathcal{P}(X,K)$.
For a generic linear subspace $L \subseteq K^\perp$, we have that 
\begin{equation}\label{eq:polarOfPolar}
    \mathcal{P}(\pi(\mathcal{P}(X, K)),L) = \textnormal{cl} \, \pi(\mathcal{P}(X, \pi^{-1}(L))). 
\end{equation}
 In particular, if $X$ satisfies the assumptions in Lemma~\ref{lem:projBirational}, then we have for generic $K$ that
\begin{equation}
    \delta_j( \pi(\mathcal{P}(X, K))) = \delta_{j+k}(X). 
\end{equation}
\end{lem}
\begin{proof}
We begin by proving \eqref{eq:polarOfPolar}.
If $\mathcal{P}(X, K) = \emptyset$, then we also have that $\mathcal{P}(X, \pi^{-1}(L)) = \emptyset$ due to $K \subseteq \pi^{-1}(L)$, and so \eqref{eq:polarOfPolar} holds trivially. 
Hence, from now on we assume that $\mathcal{P}(X, K) \neq \emptyset$.
Consider a generic $x \in \mathcal{P}(X, \pi^{-1}(L))$. Then, $\tang{x}{X} + \pi^{-1}(L) $ is a proper subspace of $\RR^n$. Since $K \subseteq \pi^{-1}(L) $, we conclude that $x \in \mathcal{P}(X, K)$. Due to the genericity of $L$, we have that $x$ is a generic point in $\mathcal{P}(X,K)$. Thus,  $\tang{\pi(x)}{\pi(\mathcal{P}(X, K))} = \pi(\tang{x}{\mathcal{P}(X,K)})$, which implies that 
\begin{equation}\label{eq:polsub}
\tang{\pi(x)}{\pi(\mathcal{P}(X, K))} + L = \pi(\tang{x}{\mathcal{P}(X,K)} + \pi^{-1}(L))    
\end{equation}
is a proper subspace of $K^\perp$. In conclusion, $\pi(x) \in \mathcal{P}(\pi(\mathcal{P}(X,K)), L)$.  

Now, consider a generic $y \in \mathcal{P}(\pi(\mathcal{P}(X,K)), L)$, and a generic preimage $x \in \mathcal{P}(X,K)$ under $\pi$. 
Due to the genericity of $L$, we have again that $x$ is a generic point on $\mathcal{P}(X,K)$, and so \eqref{eq:polsub} holds. Hence, $x \in \mathcal{P}(\mathcal{P}(X,K), \pi^{-1}(L))$, which implies that 
\begin{equation}
    \dim  \tang{x}{\mathcal{P}(X,K)} \cap \pi^{-1}(L) \geq \dim  \tang{x}{\mathcal{P}(X,K)}  +  \dim  \pi^{-1}(L) - (n-1) = \dim  L.
\end{equation}
The latter equality holds, since we have by \eqref{eq:polarOurVsClassical} and Lemma~\ref{lem:dimensioPolar} that $\dim  \mathcal{P}(X,K) = n-k-1$.
Moreover, $\dim  \tang{x}{X} \cap K \geq d + k -n + 1$, where $d := \dim  X$.
Furthermore, the assumption $\dim  \pi (\mathcal{P}(X,K)) = \dim   \mathcal{P}(X,K)$ and the fact that $x$ is a generic point on $\mathcal{P}(X,K)$ yield that $\tang{x}{\mathcal{P}(X, K) }  \cap K = \{ 0\}$. Putting everything together,
\begin{equation}
\begin{aligned}
    \dim  \tang{x}{X} \cap \pi^{-1}(L) &\geq \dim  \tang{x}{\mathcal{P}(X,K)} \cap \pi^{-1}(L) + \dim   \tang{x}{X} \cap K  \\
    &\geq \dim  L + d + k - n + 1 = \dim  \pi^{-1}(L) + d - n + 1.
\end{aligned}
\end{equation}
We conclude that $x \in \mathcal{P}(X, \pi^{-1}(L))$, as desired. 

Finally, as we have now proven \eqref{eq:polarOfPolar},  we observe for $j := \dim  L$ that 
$$\delta_j( \pi(\mathcal{P}(X, K)))
= \degg \mathcal{P}(\pi(\mathcal{P}(X, K)), L)
= \degg \pi(\mathcal{P}(X, \pi^{-1}(L)))
.$$
Since $\pi$ restricted to $\mathcal{P}(X,K)$ is birational and its kernel $K$ intersects $\mathcal{P}(X,K)$ only at the origin (by Lemma~\ref{lem:projBirational} and its proof), the map $\pi$ does not change the degree of $\mathcal{P}(X,K)$. Due to the genericity of $L$, the same holds for $\mathcal{P}(X,\pi^{-1}(L))$, i.e., 
$\degg \pi(\mathcal{P}(X, \pi^{-1}(L))) = \degg \mathcal{P}(X, \pi^{-1}(L))$.
The latter equals $\delta_{j+k}(X)$ due to the genericity of $K$.
\end{proof}

\begin{proof}[Proof of Theorem~\ref{thm:polardegs}] 
Recall that in the setting of affine cones where $k \geq n - \dim  X$ we have that $\pi(\mathcal{P} (X,K)) = \mathrm{Br}(\pi |_X)$. Hence, by \eqref{eq:polarFormula} and Lemma~\ref{lemm:polarrecurs}, we conclude
    \begin{align}
        \textnormal{gEDD} (\mathrm{Br}(\pi |_X)) = 
        \sum_{j=0}^{n-k-1} \delta_j (\mathrm{Br}(\pi |_X))
        = \sum_{j=k}^{n-1} \delta_{j} (X).
    \end{align}
This is the stated formula for the generic EDD in terms of polar classes.
\end{proof}

\bigskip
\section{Case Studies} \label{sec:examples}
In this section, we extensively discuss two examples arising in machine learning.
In both examples, the arising neuromanifolds typically have singular points. As before, we mainly focus on the smooth critical points of quadratic optimization problems, but we showcase that our critical-point computations can be done on each singular stratum as well.

\subsection{Linear Neural Networks}\label{sec:determinant}
Motivated by Example \ref{exmp:determinantal}, we consider the case of the \emph{determinantal variety}, which is the neuromanifold of a multilayer perceptron with identity activation function $\sigma(x)=x$ and without bias vectors. Given $r, \numm{in}, \numm{out}$, we denote by $X$ the variety of $\numm{out}\times \numm{in}$ matrices of rank at most $r$.
If $0 < r < \min\{\numm{in}, \numm{out}\}$, the variety $X$ is singular. Indeed, its singular locus is the set of matrices of rank at most $r-1$. Hence, the results from this section determine the critical points on each singular stratum (with the ranks ranging from $0$ through $r$) of $X$.

In the following, we use the notation of Section~\ref{sec:machinepersp}. 
Since the networks we consider in this section are linear, the Veronese embedding $\nu_D$ is not needed.
For the purpose of optimizing the loss \eqref{eq:losss}, given the dataset $\mathcal{S}$, we first prove that the matrix $AA^\top \in \RR^{\numm{in} \times \numm{in}}$ can be assumed to be diagonal with entries in $\{ 0,1 \}$. As a consequence, the quadric $Q = I_{\numm{out}}\otimes AA^\top$ from \eqref{special_Q} is a (potentially degenerate) standard quadric, still with a tensor structure. Thus, the scenario discussed in this section is radically different from the previous sections, where the quadratic form $Q$ was assumed to be generic. 
\begin{lem}\label{lem_rewrite_L(W)}
Let $A=U\Sigma V^\top $ be a singular value decomposition of $A$, where $U$ and $V$ are orthogonal matrices and the entries of the diagonal matrix $\Sigma$ are in non-ascending order.
Write $\Sigma = DP$, where $D\in\mathbb R^{\numm{in}\times \numm{in}}$ is an invertible diagonal matrix and $P \in \RR^{\numm{in} \times |\mathcal{S}|}$ is a (generally non-square) diagonal matrix whose first $\rank(A)$ entries are $1$ and all others $0$. 
 Then, we have  $\mathcal L(W) = \| WUD - R\|_{PP^\top}^2 +  \textnormal{const},$ where $R:=B V\Sigma^{\dagger}D\in\mathbb R^{\numm{out} \times \numm{in}}$.
In particular, the optimization problem (\ref{eq:lossopt}) is equivalent to solving
\begin{equation}\label{new_problem}\arg\min_{\hat W\in X }\| \hat W - R\|_{PP^\top}^2, \end{equation}
in the sense that their minimizers and smooth critical points are in bijection via multiplication with the invertible matrix $UD$.
\end{lem}
\begin{proof}
For a matrix $M$, we have 
$\Vert M\Vert_{AA^\top} = \Vert MUDP V^\top \Vert = \Vert MUDP \Vert = \Vert MUD \Vert_{PP^\top} .$
Moreover, the Moore--Penrose inverse of $A$ is  $A^\dagger = V\Sigma^{\dagger} U^\top $ and so, using (\ref{eq:lossopt}), we conclude that
\begin{align*}
\mathcal L(W) =  \| W - B A^{\dagger}  \|_{AA^\top}^2 +  \textnormal{const}
&=  \| W - B V\Sigma^{\dagger} U^\top   \|_{AA^\top}^2 +  \textnormal{const}\\
&= \| (W - B V\Sigma^{\dagger}U^\top )UD \|_{PP^\top} ^2 +  \textnormal{const}\\
&= \| WUD - B V\Sigma^{\dagger}D\|_{PP^\top}^2 +  \textnormal{const}.
\end{align*}
Since $\reg{X} \to \reg{X}, W \mapsto WUD$ is a diffeomorphism, it yields a 1-to-1 correspondence between the critical points of 
$\reg{X} \to \RR, W \mapsto \Vert WUD-R \Vert_{PP^\top}^2$ and $\reg{X} \to \RR, \hat W \mapsto \Vert \hat W-R \Vert_{PP^\top}^2$.
\end{proof}

A subtlety is that the matrix $R\in \mathbb R^{\numm{out} \times \numm{in}}$ in  Lemma \ref{lem_rewrite_L(W)} has rank bounded as
\begin{equation}\rank(R) \leq p:=\min\{\numm{in}, \numm{out}, \numm{data}\}, \quad \text{ where } \numm{data} := |\mathcal{S}|,\end{equation}
since $R$ is the product of a $(\numm{out}\times \numm{data})$-matrix and a $(\numm{data}\times \numm{in})$-matrix. Thus, if the dataset size satisfies $\numm{data} < \min\{\numm{in}, \numm{out}\}$, we cannot assume that $R$ is a generic matrix, even if we assume that $A$ and $B$ are generic.
Hence, three ranks are involved in the optimization problem: the rank of $r$ of the architecture, the rank $\rank(A) = \rank(P)$ of the seminorm, and the rank $p$ of the data matrix $R$. When the dataset $\mathcal{S}$ is generic and $\numm{data}\leq\numm{in}$, we have that $\rank(A) =\numm{data}$, which we will assume from now on. We study several cases depending on the specific ordering of these ranks. We first recall a result describing the tangent space of~$X$ -- for a proof, see, e.g., \cite{Arbarello1985}.

\begin{fact}\label{lem:det_tangent_space}
Let $W\in X$ of rank $r$. Then, $W$ is a smooth point in $X$ and 
\begin{equation}
\label{eq:tensortang}
\begin{aligned}
\tang{W}{X}  &= \spann\{uv^\top \in \mathbb R^{\numm{out} \times \numm{in}} \mid  u \in \cspan(W)\ \textnormal{ or }\  v \in \rspan(W)\} \\
&=\RR^{\numm{out}} \otimes \rspan(W) + \cspan(W) \otimes \RR^{\numm{in}},
\end{aligned}
\end{equation}
where $\cspan$ and $\rspan$ denote the span of the columns and of the rows, respectively. 
\end{fact}

Given $W,R \in\mathbb R^{\numm{out}\times \numm{in}}$, let $W'$ and $R'$ denote their left $\numm{out}\times \numm{data}$ blocks. The quadric function in the optimization problem (\ref{new_problem}) can then be written as 
\begin{equation}\label{eq:wprp}
\Vert W-R\Vert_{PP^\top} = \Vert W'-R'\Vert_{\textnormal{Frob}}.
\end{equation}
Note that the genericity of $\mathcal{S}$ implies that $R'$ is a generic $\numm{out} \times \numm{data}$ matrix and that $R = (R' \mid 0)$; cf.\ Proposition~\ref{prop:largeData}.

Given $j \leq r$, let $C_j$ be the set of all critical points on the determinantal variety of matrices of rank at most $r - j$ with respect to the Euclidean distance to $R'$.
Due to the Eckart--Young Theorem (see, e.g., \cite[Example 2.3]{draisma2016euclidean}), if $r-j\leq p=\min\{\numm{data}, \numm{out}\}$, then $C_j$ consists of
\begin{equation} \label{eq:CkCard}
|C_j| = {p \choose r-j}
\end{equation}
many points obtained from choosing $r-j$ of the $p$ singular values of $R'$.
We then define for all $\max\{0,r-p\}\leq j\leq r$:
\begin{align} \label{eq:Hk}
\begin{split}
    H_j:=\big\{ \begin{pmatrix} W' & W'' \end{pmatrix} \in X \,\mid\, W'&\in C_j,\ \cspan(W'')\subset \cspan (R')^\perp\oplus\cspan (W'),\\
&\textnormal{ and the }\cspan (R')^\perp\textnormal{ part of }W''\textnormal{ has rank } j\big\}. \end{split}
\end{align}
The definition of $H_j$ has the following motivation. Suppose that $W \in X$ is a critical point whose left block is $W'$, and thus $W'$ is a low-rank matrix approximation of $R'$. Then, if $W'$ has rank $r-j$, the right block $W''$ of $W$ must be chosen in such a way that $W=\begin{pmatrix} W' & W'' \end{pmatrix}$ has rank~$r$. The degrees of freedom for this are choosing the columns of $W''$ from $\cspan (R')^\perp\oplus\cspan (W')$. Here, directions in $\cspan (R')^\perp$ increase the rank, while directions in $\cspan (W')$ do not change the rank. Directions outside of $\cspan (W')$ contribute to the local perturbations of $W'$ in the tangent space and the directions orthogonal to $R'$ are the only ones which do not threaten criticality. If $p<r$, then we must additionally define for $j < r - p$:
\begin{equation}
\label{eq:hkdege}
    H_j:=\big\{ \begin{pmatrix} R' & W'' \end{pmatrix}\in X \big\}.
\end{equation}
The definition of these sets is independent of $j$; the indexing is exclusively for compatibility with the upcoming Theorem~\ref{thm:degenerateDetEDD}. This $H_j$ is equivalent to the global minimum component seen in Theorem~\ref{prop:generalXlargeK}. Note that $H_j$ in \eqref{eq:Hk} is not generally closed for $j>\max\{0,r-p\}$; its closure will allow $W''$ to have less than maximal rank. 
Moreover, $H_j$ in \eqref{eq:Hk} is empty whenever $j > \dim  \cspan (R')^\top = \numm{out}-p$ or $j > \numm{in}-\numm{data}.$
If non-empty, $H_j$ consists of 
$|C_j|$
connected components (cf.\ \eqref{eq:CkCard}), one for each possible $W'$. Each component of $H_j$ has dimension
\begin{equation}
\begin{split}
\dim (H_j)&=j\cdot (\dim \cspan (R')^\perp\oplus \cspan (W'))+r\cdot(\numm{in}-\numm{data}-j)\\
&=j\cdot (\numm{out}-p+r-j)+r\cdot (\numm{in}-\numm{data}-j) \quad\quad\textnormal{ if } j\geq r-p
\end{split}
\end{equation}
and 
\begin{equation}
\begin{split}
\dim (H_j)&=\numm{out}\cdot (r-p) + r\cdot (\numm{in}-\numm{data}-r+p)\quad\quad\quad\quad\textnormal{ if } j<r-p.
\end{split}
\end{equation}
This can be seen by counting the degrees of freedom when choosing the columns of $W''$; there are $\dim (\cspan (R')^\perp\oplus \cspan (W'))$ degrees of freedom when choosing each of the first $j$ columns, at which point $W$ achieves its maximum rank $r$, and then there are $r$ degrees of freedom when choosing each of the remaining columns.
\begin{thm}\label{thm:degenerateDetEDD}
    Given a generic $\numm{out}\times \numm{data}$ matrix $R'$ and $R = (R' \mid 0)$, the critical points of \eqref{new_problem} are precisely
    \begin{equation}\bigcup\limits_{j=\max\{0,r-\numm{data}\}}^{r}H_j.
    \end{equation}
\end{thm}
This result stands in contrast to Theorems \ref{prop:generalXsmallK} and  \ref{prop:generalXlargeK}, which assume that the kernel of the quadric~$Q$  intersects the variety generically. In this scenario, $Q$ has a specific, non-generic structure. To evaluate this theorem through this lens, we note that the setting of Section \ref{sec:mildlyDegenerate} corresponds to the case where the rank $p$ of a generic $R'$ is strictly larger than  $r$, while the setting of Section \ref{sec:veryDegenerate} corresponds to the case $r-p\geq 0$, in which case we also have the critical component $H_j$ in \eqref{eq:hkdege}.
\begin{exmp}\label{exmp_crit_pts_rank-one}
We work out Theorem \ref{thm:degenerateDetEDD} in the case of $2\times 3$ matrices of rank $r=1$. In this case, we have the following table of critical point sets 

\begin{center}
\begin{tabular}{c|c|c}
$n_\mathrm{data}=|\mathcal{S}|$ & $\dim \ker Q$ & complex critical point set \\ \hline
$\geq 3$ & $0$ & $2$ points\\
$2$ & $2$ & $2$ points\\
$1$ & $4$ & $2$ planes (minus the origin)\\
\end{tabular}
\end{center}
For $n_\mathrm{data}\geq 3$, $Q$ gives the standard Frobenius metric on $2\times 3$ matrices. The Eckart--Young theorem implies that we have $2$ critical points.
For $n_\mathrm{data}=2$, the critical locus is $H_0$ and it consists of $2$ points ($H_1$ is empty because $\cspan(R')^\perp=0$). Finally, for $n_\mathrm{data}=1$, the critical point set is a union of the two planes $H_0 = \{ \begin{pmatrix} R' & \lambda_1R' & \lambda_2R'\end{pmatrix} \mid \lambda \in \RR^2 \}$ and $H_1 = \{ \begin{pmatrix} 0& \lambda_1 r^\perp & \lambda_2 r^\perp\end{pmatrix} \mid \lambda \in \RR^2, \lambda \neq 0\}$, where $0 \neq r^\perp\in\cspan(R')^\perp.$
\end{exmp}

\color{black}
\begin{proof}[Proof of Theorem \ref{thm:degenerateDetEDD}]
Keep notation as above. 
To begin, we note that if $p<r$ then, for $j<r-p$, the set $H_j$ defined in \eqref{eq:hkdege} is nonempty and consists of all zero-loss approximations. In particular, this means $H_j$ is contained in the critical locus. 

We now assume that $j\geq r-p$.
We wish to show that $W = \begin{pmatrix} W' & W'' \end{pmatrix}  \in X$ is a critical point if, and only if, $W \in H_j$, where $\rank(W') = r - j$.
From \eqref{eq:wprp} it follows that $W$ is a critical point if, and only if, $W' - R'$ is orthogonal (with respect to the standard Frobenius scalar product) to the projected tangent space $(\tang{W}{X}) \cdot P = \{MP \mid M \in \tang{W}{X} \}$ (with slight abuse of notation as the matrices in the  tangent space  contain redundant vanishing entries). From \eqref{eq:tensortang} we deduce that 
\begin{equation}
\begin{aligned}
    \tang{W}{X} \cdot P &=  \RR^{\numm{out}} \otimes (\rspan(W) \cdot P) + \cspan(W) \otimes (\RR^{\numm{in}} \cdot P) \\
    &=\RR^{\numm{out}} \otimes \rspan(W') + \cspan(W) \otimes \RR^{\numm{data}}.
\end{aligned}
\end{equation}
The orthogonal complement of the above expression is: 
\begin{equation}
\label{eq:orthotensor}
    \begin{aligned}
       (\tang{W}{X} \cdot P)^\perp &= \left(\RR^{\numm{out}} \otimes \rspan(W')\right)^\perp  \cap \left(\cspan(W) \otimes \RR^{\numm{data}}\right)^\perp \\
       &=  \cspan(W)^\perp  \otimes \rspan(W')^\perp. 
    \end{aligned}
\end{equation}
 Again by \eqref{eq:tensortang}, the tangent space at $W'$ to the determinantal variety of matrices of rank at most $r-j$ coincides with $\RR^{\numm{out}} \otimes \rspan(W') + \cspan(W') \otimes \RR^{\numm{data}}$, and is therefore contained in~$\tang{W}{X}$. Together with \eqref{eq:orthotensor}, this implies that 
 $$\text{$W$ is a critical point } 
 \quad \Longleftrightarrow \quad 
 \text{ $W' \in C_j$,\; and\; $\cspan(W' - R') \subseteq \cspan(W)^\perp$.}$$ The second of the two conditions reduces to $\cspan(W' - R') \subseteq \cspan(W'')^\perp$ since, when $W' \in C_j$, we already have $\cspan(W' - R') \subseteq \cspan(W')^\perp$. Taking orthogonal complements, we can rephrase the condition for criticality as follows:
 \begin{equation}\label{critical_condition}
     W = \begin{pmatrix} W' & W'' \end{pmatrix}  \in X \text{ is critical } \; \Longleftrightarrow \; W' \in C_j \text{ \& } \cspan(W'') \subseteq \cspan(W' - R')^\perp.
 \end{equation} 

The rest of the proof is based on this equivalence. We will show below that 
\begin{equation}
\label{eq:finalconc}
\cspan(W' - R')^\perp = \cspan (R')^\perp \oplus \cspan (W').
\end{equation}
From \eqref{eq:finalconc} we conclude the proof as follows:
Since we force $\rank(W') = r-j$, as explained after \eqref{eq:Hk}, the $\cspan (R')^\perp$ part of $W''$ is forced to have rank $j$.
Therefore, \eqref{critical_condition} is equivalent to the condition for $H_j$ in \eqref{eq:Hk}.

Now, it remains to show \eqref{eq:finalconc}.
From the Eckart--Young Theorem  it follows that, if $R'=T_1\Sigma T_2$ is a singular value decomposition, then we must have $W'=T_1\Sigma'T_2$ where the singular values satisfy $\sigma'_i=\sigma_i$ for $r-j$ indices $i$ and all other $\sigma_i'=0$ \cite[Ex. 2.3]{draisma2016euclidean}. In particular, $W' - R' = T_1 (\Sigma' - \Sigma)T_2$. By the genericity of~$R'$,  there are $p$ non-vanishing and distinct singular values $\sigma_i$. Without loss of generality, we assume that the first $p$ $\sigma_i$ are non-zero, and that the first $r-j$ $\sigma'_i$ are nonzero. Since $T_2$ has full rank, multiplying by it on the right preserves the column span. Moreover, since $T_1$ is orthogonal, it commutes with taking the column span, and with taking orthogonal complements:   
\begin{equation}
\label{eq:zeroconc}
\begin{aligned}
\cspan(W' - R')^\perp &= (T_1 \cdot \cspan( \Sigma - \Sigma'))^\perp\\ &= T_1 \cdot \cspan( \Sigma - \Sigma')^\perp. 
\end{aligned}
\end{equation}
Note that $\cspan( \Sigma - \Sigma')^\perp$ is  the space of vectors with vanishing entries from index $r - j + 1$ to~$p$, and can be written as $\cspan(\Sigma)^\perp \oplus \cspan(\Sigma')$. The right-hand side of \eqref{eq:zeroconc} reduces to:
\begin{equation}
\begin{aligned}
    T_1 \cdot (\cspan(\Sigma)^\perp \oplus \cspan(\Sigma')) &= (T_1 \cdot \cspan(\Sigma))^\perp \oplus (T_1 \cdot \cspan(\Sigma')) \\  
    &= \cspan (R')^\perp \oplus \cspan (W'),
\end{aligned}
\end{equation}~as desired.  
\end{proof}

\subsection{A Self-Attention Mechanism} \label{sec:attention}
We consider the self-attention mechanism from Example~\ref{ex:attention}, where $e'=1$ and $e=t=a=2$, and explain the missing details. 

\paragraph{The neuromanifold and its ambient space.}
The neuromanifold $X$ is a subset of the space of cubic functions $\RR^{2 \times 2} \to \RR^{1 \times 2}$, which has dimension $40$. 
However, the linear span of the neuromanifold $X$ is just six-dimensional. 
To see this,  we denote by $\alpha_1$ and $\alpha_2$ the linear forms (in two variables) that take the inner product with the first and second column of the attention matrix $A$, respectively.
Similarly, $\nu$ is the linear form (also in two variables) that is the inner product with the row vector $V$.
Finally, defining $q_i := \nu \cdot \alpha_i$ and denoting by $v,w \in \RR^2$  the two columns of $M$, we can write the self-attention mechanism \eqref{eq:attention} as 
\begin{equation}
\label{eq:attentionViaQuadratics}
\begin{aligned}
    &(v, \, w)\; \mapsto \;\begin{pmatrix}
        v_1 \, q_1(v) + v_2 \, q_2(v) + v_1 \, q_1(w) + v_2 \, q_2(w) \\[.25em]
        w_1 \, q_1(w) + w_2 \, q_2(w) + w_1 \, q_1(v) + w_2 \, q_2(v)
    \end{pmatrix}^\top.
\end{aligned}
\end{equation}
Since $q_1,q_2$ are quadratic forms in two variables, together they have six coefficients. These coefficients can be read off from \eqref{eq:attentionViaQuadratics} (e.g., from the last two summands), and all coefficients of the monomial terms appearing in \eqref{eq:attentionViaQuadratics} are linear combinations of the coefficients of the $q_i$. 
This shows that the neuromanifold $X$ spans an $\RR^6$ (i.e., the linear span of the points on $X$ is $6$-dimensional) and that we can take the coefficients of the $q_i$ as its coordinates.
In those coordinates, the  neuromanifold $X$ is the set of those pairs of quadratic forms $(q_1,q_2)$ that have a common real linear factor:
\begin{equation} \label{eq:attentionParamRes}
	X = \left\{ (\nu \cdot \alpha_1, \nu \cdot \alpha_2) \;\big|\; \alpha_1,\alpha_2,\nu \in \RR[x,y]_1 \right\} \subseteq \RR[x,y]_2^2 \cong \RR^6.
\end{equation}

\paragraph{Semialgebraic description.}
 From \eqref{eq:attentionParamRes}, we easily see why the neuromanifold $X$ is not Zariski closed, as explained in Example~\ref{ex:attentionInequ}: 
 Its Zariski closure $\bar X$ is the hypersurface in $\RR[x,y]_2^2 \cong \RR^6$ that is cut out by the resultant of the pair of quadratic forms $(q_1,q_2)$. In particular, 
 $$\dim \bar X = 5.$$
 The resultant is shown in~\eqref{attention_poly}, where $q_1 = c_1x^2+c_2xy+c_3y^2$ and $q_1=c_4x^2+c_5xy+c_6y^2$.
The zero locus $\bar X$ of the resultant consists of all pairs of real quadratic forms with a common linear factor. 
If that common factor is complex and non-real, the two quadratic forms are equal up to scaling by a constant. 
Pairs of quadratic forms with such a non-real common factor do not lie in the neuromanifold $X$; they correspond to the dashed curve segment in Figure~\ref{fig:attention}.
These pairs $(q_1,q_2)$ can be distinguished from the pairs in $X$ by that the discriminant of the quadratic forms $q_i$ is negative. 
In other words, a semialgebraic description of the neuromanifold $X$ is given by the resultant in \eqref{attention_poly} being zero and the discriminant of both quadratic forms $q_1,q_2$ being non-negative, as shown in \eqref{eq:attentionInequalities}. 

\paragraph{Singularities and boundary.}
We have also mentioned in Example~\ref{ex:attentionSing} that the variety $\bar X$ is singular. 
Its singular locus consists of all pairs of quadratic forms $(q_1,q_2)$ that are equal up to scaling by a constant \cite[Thm.~3.4]{henry2024geometry}.
Such proportional quadratic forms with positive discriminant correspond to the solid orange curve segment in Figure~\ref{fig:attention}.
When the discriminant is zero (meaning that the proportional quadratic forms are squares of a linear form), we obtain the Euclidean boundary of $X$ inside $\bar X$, which -- as described in Example~\ref{ex:attentionInequ} -- is the points in Figure~\ref{fig:attention} where the solid orange curve segment goes over into the dashed one.

\paragraph{Quadratic optimization.}
We now discuss the table at the end of Example~\ref{ex:attention} and see that it indeed aligns with our results. The table lists $\crit_{\bar X, Q}(u)$ for a generic $u \in \RR^6$ and a generic quadratic form on $\RR^6$ with $k$-dimensional kernel.
We computed its entries with \texttt{Macaulay2} \cite{M2}.
The entry for $k=0$ shows  $\textnormal{gEDD}(\bar X)=14$.
For $k=1$, we get the same number of critical points, but they come in two types:
1) Four critical points are zero-loss solutions (the number four arises from intersecting the resultant in \eqref{attention_poly} that has degree four with a generic affine line that is the translated kernel of the quadric).
2) Ten critical points lie on the ramification locus \eqref{eq:Ram}.
Similarly, for $k \geq 2$, some critical points are zero-loss solutions; they form a $(k-1)$-dimensional intersection of a kernel translate with the hypersurface $\bar X$.
For $k=2$, there are four additional critical points that are located on the ramification locus \eqref{eq:Ram}.
However, for $k \geq 3$, the ramification locus in \eqref{eq:Ram} is empty due to the following geometric reason:
As the resultant in \eqref{attention_poly} is homogeneous, we can view $\bar X$ as a projective variety in $\PP^5$. 
Its dual variety $\bar X^\vee$ is a surface in $(\PP^5)^\ast$. 
(In fact, the dual $\bar X^\vee$ is isomorphic to the set of pairs of proportional quadratic forms that are squares of linear forms. Observe that the real locus of that set is the boundary of $X$ as described above. 
This is a particular instance of the dualities and resultants investigated in the classical textbook \cite{gkz}.) 
Since the Zariski closure of $\tau(X)$  and $\bar X^\vee$ are isomorphic, simply by projectivizing each tangent hyperplane of $X$, we have that $\dim \tau(X) =2$.
Now it follows from \eqref{eq:Ram} and Lemma~\ref{lem:dimensioPolar} (cf. also the right-most condition in Lemma~\ref{lem:dimEsignificant})
 that the ramification locus is empty whenever $k \geq 3$.

However, as discussed in Section \ref{sec:machinepersp}, quadratic forms representing the mean-squared error loss on polynomial models have the special and non-generic structure \eqref{special_Q}. We saw in the last section, where we discussed the example of a determinantal variety, that the number of critical points can deviate from the generic count in Theorem \ref{main_informal}. The same happens here for the self-attention mechanism. 
We consider a minimization of the mean-squared error loss \eqref{eq:losss}, given some data $\mathcal{S}$. 
For generic $\mathcal{S}$ of fixed cardinality, the rank of the resulting quadric \eqref{special_Q} is, due to its tensor struture, $\mathrm{rank}(Q) = \min \{ 6, 2 \cdot |\mathcal{S}| \}$.
Thus, the dimension of the kernel ranges in $k \in \{ 0,2,4\}$.
Using \texttt{Macaulay2}, we compute the set of complex critical points:

\begin{center}
\begin{tabular}{c|c|c}
$|\mathcal{S}|$&$k = \dim  K$ & complex critical point set \\ \hline
$\geq 3$ & $0$ & $14$ points\\
$2$ & $2$ &  a curve and two lines \\
$1$ & $4$ &a $3$-dimensional subvariety\\
\end{tabular}
\end{center}

This agrees with the cases of generic quadrics with $k$-dimensional kernel listed in the table in Example~\ref{ex:attention}, except that the 4 points on the ramification locus for $k=2$ become instead the two lines in the table above. These two lines are contained in the ramification locus. This is because the 
kernel of the special quadric $Q$ is not in general position with respect to the variety $\bar X$, causing the ramification locus to be one dimension greater than expected. 
The geometry in this case can be imagined as in the bottom two pictures in Figure \ref{fig:curveEllipsoid} -- except that, instead of isolated critical points, there are two lines of critical points on the ramification locus. The loss is constant along each of those lines.

\paragraph{Quadratic optimization on the singular locus.} 
As pointed out in Example \ref{ex:attentionSing}, the singular locus of $\bar X$ is the determinantal variety $Y \subseteq \RR^6$ of $2 \times 3$ matrices of rank at most one.
Using \texttt{Macaulay2}, we computed $\crit_{Y,Q}(u)$ for a generic point $u \in \RR^6$ and a quadratic form $Q$ on $\RR^6$ that either is  generic with $k$-dimensional kernel or represents a mean-squared error given some generic data $\mathcal{S}$ of fixed cardinality. 
Thus, the following two tables are analogous to the tables at the end of Example~\ref{ex:attention} and the table above:
\begin{table}[h]
    \begin{subtable}[t]{0.45\textwidth}
        \vspace{0pt}
        \centering
        \caption{generic $Q$ with $k$-dimensional kernel}
       \label{tab:critSingGenQ}
        \begin{tabular}{c|c}
$k$ & complex critical point set \\ \hline
$0$ & $10$ points\\
$1$ & $10$ points\\
$2$ & $7+3$ points \\
$3$ & $3$ points $+$ a curve \\
$4$ &a surface\\
$5$ & a $3$-dimensional subvariety\\
\end{tabular}
\end{subtable}
\hfill
\begin{subtable}[t]{0.45\textwidth}
    \vspace{0pt}
    \centering
    \caption{mean squared error loss for generic data $\mathcal{S}$}
    \label{tab:critSingMSE}
    \begin{tabular}{c|c|c}
$|\mathcal{S}|$&$k$ & complex critical point set \\ \hline
$\geq 3$ & $0$ & $10$ points\\
$2$ & $2$ &  $6$ lines\footnote{The \texttt{Macaulay2} computation yields a curve of degree $6$. This curve is irreducible over $\mathbb Q$. Computing a numerical irreducible decomposition (NID) \cite{breiding2026numericalirreducibledecompositionjulia} reveals that over $\mathbb C$ this curve is a union of $6$ lines. Note that NID is a numerical computation.}\\
$1$ & $4$ & 2 planes\\
\end{tabular}
    \end{subtable}
    \medskip
    \caption{$\crit_{Y,Q}(u)$ for a generic point $u \in \RR^6$, where $Y$ is the singular locus of $X.$}
    \label{tab:temps}
\end{table}

Table \ref{tab:critSingGenQ} is aligned with our main result Theorem \ref{main_informal}:
For generic quadrics with $k$-dimensional kernel, the phase transition between the case of mildly degenerate and significantly degenerate quadrics occurs for $k=2$.
In that case, 3 of the 10 critical points are zero-loss solutions, while the other 7 are contained in the ramification locus.
When the kernel dimension $k$ rises to three, there are only 3 critical points on the ramification locus, while there is a curve of zero-loss solutions. For $k \geq 4$, the ramification locus is empty and all critical points are zero-loss solutions.

Table \ref{tab:critSingMSE} shows that the mean-squared error has a rather different behavior from  generic quadrics. For instance,  in the case of a single data point (i.e., $|\mathcal{S}=1|$, which gives rise to a quadric with kernel of dimension $k=4$), one of the two critical planes consists of zero-loss solutions, while the other plane is contained in the ramification locus (which is in particular non-empty, in contrast to the case of generic quadrics with four-dimensional kernel). Along the latter plane, the loss is constant, but generically non-zero.

Both tables should be compared with the table in Example \ref{exmp_crit_pts_rank-one}, where we list the critical point set for quadrics representing the mean squared error loss on $Y$, when the points in $Y$ are interpreted as linear maps (i.e., there, $Y$ is the neuromanifold of a linear network). In contrast, in this section, the points in $Y$ represent cubic functions, as we consider here the self-attention mechanism \eqref{eq:attentionViaQuadratics}.

All our \texttt{Macaulay2} computations can be found at \url{https://github.com/giovanni-marchetti/LinearIndepCode/tree/main}.

\bigskip
\section{Conclusion}
This work is motivated by deep learning, where a mean-squared error loss with fewer data points than parameters results in a degenerate quadratic loss function. We count the (complex) critical points of a degenerate quadratic loss $Q$ defined on an algebraic variety $X\subset \mathbb R^n$. When the quadric is not degenerate, this count is known as the Euclidean distance degree \cite{draisma2016euclidean} (EDD). We generalize the EDD to the degenerate setting.

Theorems \ref{prop:generalXsmallK} and \ref{prop:generalXlargeK} prove that, for a generic degenerate quadric $Q$, counting critical points falls into two regimes depending on the dimension $k=\dim K$ of the kernel $K$ of $Q$ relative to the dimension $d=\dim X$. When the quadric is mildly degenerate ($k<n-d$), then critical points are in bijection with critical points on $Y= \textnormal{cl}\, \pi(X)$, where $\pi$ is the orthogonal projection onto orthogonal complement of~$K$.  On the other hand, when the quadric is significantly degenerate ($k\geq n-d$), there are two classes of critical points. One class comes from points with zero loss, and the other comes from the branch locus of $\pi$.  
Furthermore, Theorem \ref{thm:polardegs} gives a formula for the number of critical points arising from the branch locus, expressed in terms of polar degrees, that extends the classical polar-degree sum formula for the EDD.

However, in deep learning the quadric has the form \eqref{special_Q} and is hence not generic. We exemplified this in Section \ref{sec:examples}, where we computed the critical points when 
minimizing the mean-squared error loss over deep linear fully-connected neural networks or a lightning self-attention mechanism.
We hope that this sparks interest in both the applied algebraic geometry and deep learning community to further study neuromanifolds and  Euclidean distance degrees for quadrics of the form \eqref{special_Q}.

\clearpage 

\section*{Acknowledgements}
G. L. M.  and K. K. were supported by the Wallenberg AI, Autonomous Systems and Software Program (WASP) funded by the Knut and Alice Wallenberg Foundation. 
K. K. was also  supported by the Swedish Foundations' Starting Grant \emph{Algebraic Vision} funded by the Ragnar Söderbergs stiftelse.
P. B. was supported by DFG, German Research Foundation -- Projektnummer 445466444.

\bigskip
\bibliographystyle{plainnat}
\bibliography{biblio}


\end{document}